\documentclass[leqno,final]{siamltex}
\usepackage{amsmath,amssymb}
\usepackage{graphicx}
\newtheorem{remark}{Remark}[section]

\newcommand{\norm}[1]{\left\Vert#1\right\Vert}
\newcommand{\abs}[1]{\left\vert#1\right\vert}

\newcommand{\hM}{\hat{M}}

\newcommand{\E}{\mathbb{E}}
\newcommand{\cE}{\mathcal{E}}
\newcommand{\cS}{\mathcal{S}}
\newcommand{\cT}{\mathcal{T}}
\newcommand{\Div}{{\rm div}}
\newcommand{\veps}{\varepsilon}
\newcommand{\Ome}{\Omega}
\newcommand{\ome}{\omega}
\newcommand{\p}{\partial}
\newcommand{\nab}{\nabla}
\renewcommand{\i}{{\rm\mathbf i}}
\renewcommand{\Im}{{\rm Im }}
\renewcommand{\Re}{{\rm Re }}

\def\be{\begin{equation}}
\def\ee{\end{equation}}
\def\br{\begin{eqnarray}}
\def\er{\end{eqnarray}}

\def\Langle{\left\langle}

\def\Rangle{\right\rangle}

\begin{document}

\title{An efficient numerical method for acoustic wave scattering in random media}
\markboth{XIAOBING FENG, JUNSHAN LIN and CODY LORTON}{Numerical Methods for
a Random Helmholtz Equation}

\author{
Xiaobing Feng\thanks{Department of Mathematics, The University of Tennessee,
Knoxville, TN 37996, U.S.A. ({\tt xfeng@math.utk.edu}) The work of
this author was partially supported by the NSF grants DMS-1016173 and DMS-1318486.}
\and
Junshan Lin\thanks{Department of Mathematics and Statistics, Auburn University,
Auburn, AL 36849, U.S.A. ({\tt jzl0097@auburn.edu}) }
\and
Cody Lorton\thanks{Department of Mathematics, The University of Tennessee,
Knoxville, TN 37996, U.S.A. ({\tt lorton@math.utk.edu}) The work of
this author was partially supported by the NSF grants DMS-1016173 and DMS-1318486.}
}

\maketitle

\begin{abstract}
This paper is concerned with developing efficient numerical methods for
acoustic wave scattering in random media which can be expressed as
random perturbations of homogeneous media. We first analyze the
random Helmholtz problem by deriving some wave-number-explicit
solution estimates. We then establish a multi-modes representation of
the solution as a power series of the perturbation parameter and analyze
its finite modes approximations. Based on this multi-modes representation,
we develop a Monte Carlo interior penalty discontinuous Galerkin (MCIP-DG)
method for approximating the mode functions, which are governed by
recursively defined nearly deterministic Helmholtz equations.
Optimal order error estimates are derived for the method and
an efficient algorithm, which is based on the LU direct solver,
is also designed for efficiently implementing the proposed
multi-modes MCIP-DG method. It is proved that the computational complexity
of the whole algorithm is comparable to that of solving one
deterministic Helmholtz problem using the LU director solver.
Numerical experiments are provided to validate the theoretical
results and to gauge the performance of the proposed numerical
method and algorithm.
\end{abstract}



\section{Introduction}\label{sec-1}

Partial differential equations with random coefficients arise naturally in the modeling of many
physical phenomena. This is due to the fact that some level of uncertainty is usually involved if the 
knowledge of the physical behavior is not complete or when noise is present in the experimental 
measurements.  In recent years, substantial progress has been made in the numerical 
approximation of such PDEs due to the significant development in computational resources. 
We refer to \cite{Babuska_Nobile_Tempone10,Babuska_Tempone_Zouraris04,Caflisch,Ronan_Sarkis, 
Xiu_Karniadakis} and references therein for more details.

In this paper, we consider the propagation of the acoustic wave in a medium where the 
wave velocity is characterized by a random process. More precisely, we study the 
approximation of the solution to the following Helmholtz problem:
\begin{alignat}{2}\label{eq1.1}
-\Delta u(\omega,\cdot) - k^2  \alpha(\ome,\cdot)^2 u(\omega,\cdot)
&= f(\omega,\cdot)  &&\qquad \mbox{in } D, \\
\p_\nu u(\omega,\cdot) +\i k \alpha(\ome,\cdot) u(\omega,\cdot)
&= 0 &&\qquad \mbox{on } \partial D, \label{eq1.2}
\end{alignat}
where $k$ is the wavenumber, and $D\subset\mathbb{R}^d \,(d=1,2,3)$ is
a convex bounded polygonal domain with boundary $\partial D$.
Let $(\Ome,\mathcal{F}, P)$ be a probability space with sample space $\Ome$,
$\sigma-$algebra $\mathcal{F}$ and probability measure $P$. For each fixed $x\in D$,
the refractive index $\alpha(\cdot,x)$ is a real-valued random variable
defined over $\Ome$. We assume that the medium is
a small random perturbation of a uniform background medium in the sense that
\begin{align}\label{eq1.3}
\alpha(\ome, \cdot):=1+\veps\eta(\omega,\cdot).
\end{align}
Here $\veps$ represents the magnitude of the random fluctuation, and 
$\eta\in L^2(\Omega,L^{\infty}(D))$ is some 
random process satisfying
\[
P\left\{\ome\in\Ome;\, \norm{\eta(\ome,\cdot)}_{L^{\infty}(D)}
\le 1 \right\}=1.
\]
For notation brevity we only consider the case that $\eta$ is
real-valued. However, we note that the results of this paper are also
valid for complex-valued $\eta$. 
On the boundary $\partial D$, a radiation boundary
condition is imposed
to absorb incoming waves \cite{em79}.
Here $\nu$ denotes the unit outward normal to $\partial D$, and $\p_\nu u$ stands for
the normal derivative of $u$. The boundary value problem \eqref{eq1.1}--\eqref{eq1.2} 
arises in the modeling of the wave propagation in complex environments, such as 
composite materials, oil reservoir and geological basins \cite{FGPS, Ishimaru}.
In such instances, it is of practical interest to characterize the uncertainty of 
the wave energy transport when the medium contains some randomness. In particular, 
we are interested in the computation of some statistics
of the wave field, e.g, the mean value of the solution $u$.

To solve stochastic (or random) partial differential equations (SPDEs) numerically, 
the simplest and most natural approach is to use the Monte Carlo method, where a 
set of independent identically distributed (i.i.d.) solutions are obtained by sampling 
the PDE coefficients, and the mean of the solution is calculated via a statistical 
average over all the sampling in the probability space \cite{Caflisch}. An alternative is 
the stochastic Galerkin method, where the SPDE is reduced to a high dimensional deterministic 
equation by expanding the random field in the equation using the Karhunen-Lo\`{e}ve or 
Wiener Chaos expansions.  We refer the reader to \cite{Babuska_Nobile_Tempone10,
Babuska_Tempone_Zouraris04,EEU,Ronan_Sarkis, Xiu_Karniadakis} for detailed discussions.
However, it is known that a brute-force Monte Carlo or stochastic Galerkin method applied 
directly to the Helmholtz equation with random coefficients is computationally prohibitive 
even for a moderate wavenumber $k$, since a large number of degrees of freedom is involved 
in the spatial discretization.  It is apparent that in such cases, the Monte Carlo method 
requires solving a PDE with many sampled coefficients, while the high 
dimensional deterministic equation associated with the stochastic Galerkin method will be 
too expensive to be solved.

In this paper, we propose an efficient numerical method for solving the Helmholtz 
problem \eqref{eq1.1}--\eqref{eq1.2} when the medium is weakly random defined 
by \eqref{eq1.3}. A multi-modes representation of the solution is derived, where each 
mode is governed by a Helmholtz equation with deterministic coefficients and a random source.
We develop a Monte Carlo interior penalty discontinuous Galerkin (MCIP-DG) method for 
approximating the mode functions.  In particular, we take the advantage that the 
coefficients of the Helmholtz equation for all the modes are identical,
hence the associated discretized equations share the same constant
coefficient matrix. Using this crucial fact, it is observed that an LU direct solver 
for the discretized equations leads to a tremendous saving in the computational costs,
since the LU decomposition matrices can be used repeatedly, and the solutions for all 
modes and all samples can be obtained in an efficient way by performing simple forward
and backward substitutions. Indeed, it turns out that the computational complexity
of the proposed algorithm is comparable to that of solving one
deterministic Helmholtz problem using the LU direct solver. 

The rest of the paper is organized as follow. A wave-number-explicit estimate for the 
solution of the random Helmholtz equation is established in Section \ref{sec-2}. 
In Section \ref{sec-3}, we introduce the multi-modes expansion of the solution as 
a power series of $\veps$ and analyze the error estimation for its finite-modes approximation.
The Monte Carlo interior penalty discontinuous Galerkin method is presented in Section 
\ref{sec-4}, where the error estimates for the approximation of each mode function is also
obtained.  In Section \ref{sec-5}, a numerical procedure for solving 
\eqref{eq1.1}--\eqref{eq1.2} is described and its computational complexity is analyzed
in detail.  In addition, we derive an optimal order error estimates for the proposed 
procedure.  Several numerical experiments are provided in Section 6 to demonstrate 
the efficiency of the method and to validate the theoretical results. 

\section{PDE analysis}\label{sec-2}

\subsection{Preliminaries}\label{sec-2.1}
Standard function and space notations are adopted in this paper. For example,
$H^s(D)$ denotes the complex-valued Sobolev space and $L^2(D)=H^0(D)$.
$(\cdot,\cdot)_S$ stands for the standard inner product on the complex-valued
$L^2(S)$ space for any subset $S$ of $D$.  $\,C$ and $c$ denote generic constants 
which are independent of $k$ and the mesh parameter $h$. We also define spaces
\begin{align}\label{eq2.14}
H^1_+(D) &:=\bigl\{v\in H^1(D);\, |\nab v|\bigr|_\Gamma \in L^2(\p D) \bigr\},\\
V &:= \big\{ v\in H^1(D); \, \Delta v\in L^2(D) \bigr\}.
\end{align}

Without loss of generality, we assume that the domain $D\subset B_R(0)$.
Throughout this paper we also assume that $D$ is a star-shaped domain
with respect to the origin in the sense that there exists a positive
constant $c_0$ such that
\[
x\cdot \nu\ge c_0 \quad\mbox{on } \partial D,
\]
Let $(\Ome, \mathcal{F}, P)$ be a probability space on which all the random
variables of this paper are defined. $\E(\cdot)$ denotes the expectation operator.
The abbreviation {\em a.s.} stands for {\em almost surely}.

As it will be needed in the late sections of the paper, in this section
we analyze the boundary value problem for the Helmholtz equation \eqref{eq1.1} 
with the following slightly more general nonhomogeneous boundary condition:
\begin{equation}\label{eq1.2a}  
\p_\nu u(\omega, \cdot) + \i k \alpha(\omega,\cdot) u(\omega,\cdot)=g(\omega,\cdot).
\end{equation}

\begin{definition}\label{def2.1}
Let $f\in L^2(\Ome, L^2(D))$ and $g \in L^2(\Ome, L^2(\partial D))$. A 
function $u\in L^2(\Ome, H^1(D))$ is called a weak solution to problem 
\eqref{eq1.1},\eqref{eq1.2a} if it satisfies the following identity:
\begin{equation}\label{eq2.1}
\int_\Ome a(u,v)\, dP = \int_\Ome \big((f, v)_D + \langle g, v \rangle_{\partial D} \big)\, dP \qquad\forall
v\in L^2(\Ome, H^1(D)),
\end{equation}
where
\begin{align}\label{eq2.2}
a(w,v) &:= \bigl(\nab w, \nab v\bigr)_D -k^2 \bigl( \alpha^2 w, v \bigr)_D
+\i k\Langle \alpha w, v\Rangle_{\p D}.
\end{align}
\end{definition}

\begin{remark}\label{rem2.1}
Using \eqref{eq2.4} below, it is easy to show that any solution $u$
of \eqref{eq1.1},\eqref{eq1.2a} satisfies $u\in L^2(\Ome, H^1_+(D)\cap V)$.
\end{remark}

\subsection{Wave-number-explicit solution estimates}\label{sec-2.2}
In this subsection we shall derive stability estimates for the solution of
problem \eqref{eq1.1},\eqref{eq1.2a} which is defined in Definition \ref{def2.1}.
Our focus is to obtain explicit dependence of the stability constants
on the wave number $k$, such wave-number-explicit stability estimates
will play a vital role in our convergence analysis in the later sections.
We note that wave-number-explicit stability estimates also play
a pivotal role in the development of numerical methods,
such as finite element and discontinuous Galerkin methods,
for deterministic reduced wave equations (cf. \cite{Feng_Wu09,Feng_Wu11}).
As a byproduct of the stability estimates, the existence and
uniqueness of solutions to problem \eqref{eq1.1},\eqref{eq1.2a}
can be conveniently established.

\begin{lemma}\label{lem2.1}
Let $u\in L^2(\Omega,H^1(D))$ be a solution of \eqref{eq1.1},\eqref{eq1.2a},
then for any $\delta_1, \delta_2>0$ and $\veps < 1$ there hold
\begin{align}\label{eq2.3}
\E(\norm{\nab u}_{L^2(D)}^2) &\leq \Bigl( k^2 (1+\veps)^2
+ \delta_1 \Bigr) \E(\norm{u}_{L^2(D)}^2) \\
& \qquad + \left( \frac{\delta_1}{2k^2(1-\veps)^2} + \frac{1}{2\delta_1} \right) \left( \E(\norm{f}_{L^2(D)}^2) + \E(\norm{g}_{L^2(\partial D)}^2)  \right),  \notag \\
\E(\norm{ u}_{L^2(\p D)}^2) &\leq 
\frac{\delta_2}{k(1-\veps)} \E(\norm{u}_{L^2(D)}^2) + \frac{1}{\delta_2k(1-\veps)} \E(\norm{f}_{L^2(D)}^2)\label{eq2.4} \\
& \qquad + \frac{1}{k^2(1-\veps)^2} \E(\norm{g}_{L^2(\partial D)}^2) \notag.
\end{align}
\end{lemma}

\begin{proof}
Setting $v=u$ in \eqref{eq2.1} yields
\[
\int_\Omega a(u,u)\, dP =\int_\Omega \big( (f,u)_D + \langle g, v \rangle_{\partial D} \big)\, dP
\]
Taking the real and imaginary parts and using the definition of $a(\cdot,\cdot)$, we get
\begin{align}\label{eq2.5}
\int_\Omega \Bigl( \|\nabla u\|_{L^2(D)}^2
- k^2(1+\varepsilon\eta)^2 \|u\|_{L^2(D)}^2 \Bigr) \,dP
&=\Re\int_\Omega \big((f, u)_D + \langle g, v \rangle_{\partial D} \big)\, dP,\\
%
k\int_\Omega \langle 1+\veps\eta, |u|^2\rangle_{\p D} \, dP&=\Im\int_\Omega \big( (f, u)_D + \langle g, v \rangle_{\partial D} \big) \, dP.
\label{eq2.6}
\end{align}
Applying the Cauchy-Schwarz inequality to \eqref{eq2.6} produces
\begin{align*}
	k(1-\veps) \E(\| u \|_{L^2(\partial D)}^2) & \leq \frac{\delta_2}{2} \E(\| u \|_{L^2(D)}^2) + \frac{1}{2 \delta_2} \E(\| f \|_{L^2(D)}^2) \\
	& \qquad + \frac{k(1-\veps)}{2} \E(\| u \|_{L^2(\partial D)}^2) + \frac{1}{2 k(1-\veps)} \E(\| g \|_{L^2(\partial D)}^2)
\end{align*}
Thus, \eqref{eq2.4} holds. Applying Cauchy-Schwarz to \eqref{eq2.5} yields
\begin{align*}
	\E(\| \nabla u \|_{L^2(D)}^2) & \leq \left(k^2(1+\veps)^2 + \frac{\delta_1}{2} \right) \E(\| u \|_{L^2(D)}^2) + \frac{1}{2 \delta_1} \E(\| f \|_{L^2(D)}^2) \\
	& \qquad + \frac{\delta_1}{2} \E(\| u \|_{L^2(\partial D)}^2) + \frac{1}{2 \delta_1} \E(\| g \|_{L^2(\partial D)}^2).
\end{align*}
To this one can apply \eqref{eq2.4} with $\delta_2 = k(1-\veps)$ and obtain \eqref{eq2.3}.
The proof is complete.
\end{proof}

\begin{lemma}\label{lem2.2}
Let $u\in L^2(\Omega,H^2(D))$, then there hold
\begin{align}\label{eq2.7}
\Re \int_\Omega \bigl(u, x\cdot\nabla u\bigr)_D \,dP
&= -\frac{d}2\int_\Omega \|u\|_{L^2(D)}^2 \,dP
+\frac12 \int_\Omega  \langle x\cdot \nu, |u|^2\rangle_{\p D} \,dP, \\
\Re \int_\Omega \bigl(\nab u, \nab(x\cdot\nabla u)\bigr)_D \,dP
&=\frac{2-d}2 \int_\Omega \|\nabla u\|_{L^2(D)}^2 \,dP \label{eq2.8} \\
&\hskip 0.8in
+\frac12 \int_\Omega \Langle x\cdot \nu, |\nabla u|^2 \Rangle_{\p D}\, dP. \nonumber
\end{align}
\end{lemma}

\begin{proof}
\eqref{eq2.7} follows immediately from applying the divergence theorem to
\[
\int_\Ome \bigl(\Div(x|u|^2), 1\bigr)_D\, dP,
\]
and the fact that $\Div (x)=d$. To show \eqref{eq2.8}, we first recall the
following differential identities \cite{Cummings_Feng06}:
\begin{align*}
\nabla\cdot(x|\nabla u|^2) &= d |\nabla u|^2 + x\cdot\nabla(|\nabla u|^2),\\
x\cdot\nabla(|\nabla u|^2) &= 2 \Re \Bigl( \nabla\cdot(\nabla u
(\overline{x\cdot\nabla u})) - \Delta u (\overline{x\cdot\nabla u}) \Bigr)
- 2 |\nabla u|^2\\
&=2\Re  \bigl(\,\nabla u \cdot \overline{\nab(x\cdot\nabla u)}\, \bigr)
- 2 |\nabla u|^2.
\end{align*}
Then \eqref{eq2.8} follows from adding the above two identities, integrating
the sum over $D\times \Omega$ and applying the divergence theorem on the left-hand
side of the resulting equation.
\end{proof}

\begin{remark}
\eqref{eq2.8} could be called a stochastic Rellich identity for the Laplacian.
\end{remark}

We are now ready to state and prove our wave-number-explicit
estimate for solutions of problem \eqref{eq1.1},\eqref{eq1.2a} defined
in Definition \ref{def2.1}.

\begin{theorem}\label{thm2.1}
Let $u\in L^2(\Omega,H^1(D))$ be a solution of \eqref{eq1.1},\eqref{eq1.2a} and
$R$ be the smallest number such that $B_R(0)$ contains the domain $D$. Then
there hold the following estimates:
\begin{align}\label{eq2.9}
\E\Bigl(\norm{u}_{L^2(D)}^2 +\norm{u}_{L^2(\p D)}^2 +c_0\norm{\nab u}_{L^2(\p D)}^2\Bigr)
&\leq C_0 \Bigl( \frac{1}{k} +\frac{1}{k^2} \Bigr)^2 M(f,g), \\
%
\E(\norm{u}_{H^1(D)}^2) &\leq C_0\Bigl(1+\dfrac{1}{k^2}\Bigr)^2 M(f,g), \label{eq2.9a}
\end{align}
provided that $\veps (2+\veps) < \gamma_0 :=\min\bigl\{1,\frac{13-2d}{2(4d-7)+25kR}\bigr\}$. 
Where $C_0$ is some positive constant independent of $k$ and $u$, and
\begin{equation}\label{eq2.9x}
M(f, g):= \E\Bigl( \norm{f}_{L^2(D)}^2 + \norm{g}_{L^2(\p D)}^2 \Bigr).
\end{equation}
Moreover, if $g\in L^2(\Ome, H^{\frac12}(D))$ and $u\in L^2(\Omega,H^2(D))$, there also holds

\begin{align}\label{eq2.9b}
\E(\norm{u}_{H^2(D)}^2) \leq C\Bigl(k+\dfrac{1}{k^2}\Bigr)^2 
\E\Bigl( \norm{f}_{L^2(D)}^2 + \norm{g}_{H^{\frac12}(\p D)}^2 \Bigr).
\end{align}

\end{theorem}

\begin{proof}
To avoid some technicalities, below we only give a proof for the case
$u\in L^2(\Omega,H^2(D))$. For the general case, $u$ needs be replaced
by its mollification $u_\rho$ at the beginning of the proof and followed by
taking the limit $\rho\to 0$ after the integration by parts is done.

Setting $v=x\cdot\nabla u$ in \eqref{eq2.1} yields
\begin{equation}\label{eq2.10}
\int_\Omega \Bigl( (\nab u, \nab v)_D - k^2 (\alpha^2 u, v)_D
+ \i k \langle \alpha u, v\rangle_{\p D} \Bigr)\, dP
= \int_\Omega \left( (f, v)_D  + \langle g,v \rangle_{\partial D}\right)\,dP.
\end{equation}
Using \eqref{eq2.7} and \eqref{eq2.8} after taking the real part
of \eqref{eq2.10} and regrouping we get
\begin{align*}
&\dfrac{dk^2}{2}\int_\Omega\|u\|_{L^2(D)}^2\,dP
= \int_\Omega \Bigl( \frac{d-2}{2} \|\nabla u\|_{L^2(D)}^2
 + k^2\veps\, \Re\bigl(\eta (2+\veps\eta), v\bigr)_D \Bigr)\, dP\\
&\hskip .6in
-\int_\Omega \Bigl( k\Im \langle (1+\veps\eta) u, v\rangle_{\p D}
+\frac12 \Langle x\cdot\nu,|\nabla u|^2 \Rangle_{\p D}
-\frac{k^2}{2} \Langle x\cdot\nu,|u|^2 \Rangle_{\p D}\Bigr)\, dP \\
&\hskip .6in + \int_{\Omega} \left(\Re (f, v)_D + \Re \langle g,v \rangle_{\partial D} \right) \, dP
\end{align*}
It then follows from Schwarz inequality and the ``star-shape" condition and the
facts that $|x|\leq R$ for $x\in D$ and $\|\eta\|_{L^\infty(D)}\leq 1$ a.s. that
\begin{align*}
\frac{dk^2}{2} \E(\norm{u}_{L^2(D)}^2)
&\leq \frac{d-2}{2} \E(\norm{\nabla u}_{L^2(D)}^2)
+ k^2 \veps R(2+\veps) \Bigl( \frac{1}{2\delta_1} \E(\norm{u}_{L^2(D)}^2) \\
&\qquad
+ \frac{\delta_1}{2} \E(\norm{\nabla u}_{L^2(D)}^2) \Bigr)
+\frac{R}{2\delta_2}\E(\norm{f}_{L^2(D)}^2)
+\frac{R\delta_2}{2} \E(\norm{\nabla u}_{L^2(D)}^2) \\
& \qquad +\frac{R}{2\delta_3}\E(\norm{g}_{L^2(\partial D)}^2)
+\frac{R\delta_3}{2} \E(\norm{\nabla u}_{L^2(\partial D)}^2) \\
&\qquad
+ \frac{kR}{\delta_4} \E(\norm{u}_{L^2(\p D)}^2)
+kR\delta_4 \E(\norm{\nabla u}_{L^2(\p D)}^2)  \\
&\qquad
-\frac{c_0}{2} \E\bigl(\norm{\nabla u}_{L^2(\p D)}^2\bigr) +
\frac{k^2 R}{2} \E\bigl(\norm{u}_{L^2(\p D)}^2\bigr).
\end{align*}
At this point, we note that $\veps(2+\veps) \leq 1$ implies $\veps \leq \frac{1}{2}$.
Setting $\delta_3 = \frac{c_0}{4R}$, $\delta_4=\frac{c_0}{8kR}$ and denoting $\gamma=\veps(2+\veps)$,
using Lemma \ref{lem2.1} we can bound right-hand side as follows:
\begin{align*}
&\frac{dk^2}{2} \E(\norm{u}_{L^2(D)}^2)  \leq \Bigl(\frac{d-2}{2}
+\frac{k^2 R\gamma \delta_1}{2} + \frac{R\delta_2}{2} \Bigr)\E(\norm{\nabla u}_{L^2(D)}^2)
+\frac{k^2 R\gamma}{2\delta_1} \E(\norm{u}_{L^2(D)}^2) \\
&\qquad\qquad
+\Bigl( \frac{8k^2R^2}{c_0}+ \frac{k^2R}{2} \Bigr) \E(\norm{u}_{L^2(\p D)}^2)
-\frac{c_0}{4} \E\bigl(\norm{\nabla u}_{L^2(\p D)}^2\bigr) \\
& \qquad \qquad 
+\frac{R}{2\delta_2} \E(\norm{f}_{L^2(D)}^2)
+ \frac{2R^2}{c_0} \E(\norm{g}_{L^2(\partial D)}^2) \\
&\qquad
\leq \Bigl(\frac{d-2}{2} +\frac{k^2 R\gamma \delta_1}{2} + \frac{R\delta_2}{2} \Bigr)
 \left( k^2 (1+\gamma) + \delta_5 \right) \E(\norm{u}_{L^2(D)}^2)
 \\
& \qquad \qquad
+ \Bigl(\frac{d-2}{2} +\frac{k^2 R\gamma \delta_1}{2} + \frac{R\delta_2}{2} \Bigr) \Bigl( \frac{2 \delta_5}{k^2} + \frac{1}{2 \delta_5} \Bigr) \left( \E(\norm{f}_{L^2(D)}^2)+  \E(\norm{g}_{L^2(\partial D)}^2) \right)
\\
&\qquad\qquad
+\Bigl( \frac{8k^2R^2}{c_0}+ \frac{k^2R}{2}  \Bigr)
\Bigl(2\delta_6 \E(\norm{ u}_{L^2(D)}^2)
+\frac{2}{k^2\delta_6} \E(\norm{f}_{L^2(D)}^2) + \frac{4}{k^2}\E(\norm{g}_{L^2(\partial D)}^2)\Bigr) \\
&\qquad\qquad
+\frac{k^2 R\gamma}{2\delta_1} \E(\norm{u}_{L^2(D)}^2)
+\frac{R}{2\delta_2} \E(\norm{f}_{L^2(D)}^2)
+\frac{2R^2}{c_0} \E(\norm{g}_{L^2(\partial D)}^2)
-\frac{c_0}{4} \E\bigl(\norm{\nabla u}_{L^2(\p D)}^2 \bigr),
\end{align*}
which is equivalent to
\begin{equation}\label{eq2.12}
c_1 \E(\norm{u}_{L^2(D)}^2) + \frac{c_0}{4} \E(\norm{\nabla u}_{L^2(\p D)}^2)
\leq c_2 \E(\norm{f}_{L^2(D)}^2),
\end{equation}
where
\begin{align*}
c_1 &:=  k^2 - \frac{d-2}2 \Bigl(k^2\gamma + \delta_5 \Bigr)
- \Bigl(\frac{k^2 R\gamma \delta_1}{2} + \frac{R\delta_2}{2} \Bigr)
\Bigl( k^2 (1+\gamma) + \delta_5 \Bigr) \\
&\qquad\quad
-\Bigl(\frac{16k^2R^2}{c_0} + k^2R \Bigr) \delta_6
- \frac{k^2 R\gamma}{2\delta_1}, \\
c_2 &:= \Bigl(\frac{d-2}2 + \frac{k^2 R\gamma \delta_1}{2} + \frac{R\delta_2}{2} \Bigr)
\Bigl( \frac{2\delta_5}{k^2} + \frac{1}{2\delta_5} \Bigr) \\
& \qquad \qquad 
+ \Bigl(\frac{32k^2R^2}{c_0} + 2k^2R \Bigr) \Bigl( \frac{2}{k^2\delta_6} + \frac{4}{k^2} \Bigr)
+ \frac{R}{2\delta_2} + \frac{2R^2}{c_0}.
\end{align*}

Let $\delta_1=\frac{1}{2k}$, $\delta_2=\frac{1}{4R}$, $\delta_5=\frac{k^2}{4}$,
and $\delta_6=\frac{1}{4\bigl( \frac{16R^2}{c_0}+ R\bigr)}$, then
\begin{align*}
c_1 &= k^2 \Bigl[ \frac{27-4d}{32} - \Bigl( \frac{4d-7}{8}
+\frac{(21+4\gamma)Rk}{16} \Bigr) \gamma \Bigr],\\
c_2 &= \Bigl( \frac{d-2}{2} + \frac{kR\gamma}{4} + \frac{1}{8} \Bigr) \Bigl( \frac{1}{2} + \frac{1}{8k^2} \Bigr) + 2R^2(1+c_0) \\
& \qquad
	+ 8 \Bigl( \frac{16R^2}{c_0} + R \Bigr) \Bigl( \frac{16R^2}{c_0} + R + 1 \Bigr)
\end{align*}
If $\gamma < \gamma_0$, it is easy to check that $c_1\geq \frac{k^2}{32}$.  Thus,
\eqref{eq2.12} infers that
\begin{equation}\label{eq2.13}
\E(\norm{u}_{0,D}^2)  + c_0 \E(\|\nab u\|_{L^2(\p D)}^2)
\leq \frac{C}{k^2}\Bigl(1+\frac{1}{k^2}\Bigr) \left( \E(\norm{f}_{L^2(D)}^2) + \E(\norm{g}_{L^2(\partial D)}^2) \right)
\end{equation}
for some constant $C>0$ independent of $k$ and $u$. This then proves \eqref{eq2.9}.

By \eqref{eq2.3} with $\delta_1= k^2$ and \eqref{eq2.13} we get
\begin{align*}
\E(\norm{u}_{H^1(D)}^2) &= \E(\norm{u}_{L^2(D)}^2) + \E(\norm{\nabla u}_{L^2(D)}^2) \\
&\leq \frac{C}{k^2}\Bigl(1+\frac{1}{k^2}\Bigr) \Bigl( \E(\norm{f}_{L^2(D)}^2) + \E(\norm{g}_{L^2(\partial D)}^2) \Bigr) \\
&\qquad \qquad
+ \left(k^2(1+\veps)^4 + k^2 \right)\E(\norm{u}_{L^2(D)}^2) \\
&\qquad \qquad 
+ \Bigl(1 + \frac{1}{2k^2} \Bigr) \Bigl( \E(\norm{f}_{L^2(D)}^2) + \E(\norm{g}_{L^2(\partial D)}^2) \Bigr) \\
&\leq C\Bigl(1+\frac{1}{k^2}\Bigr)^2 \Bigl( \E(\norm{f}_{L^2(D)}^2) + \E(\norm{g}_{L^2(\partial D)}^2) \Bigr) .
\end{align*}
Hence, \eqref{eq2.9a} holds.

Finally, it follows from the standard elliptic regularity theory for Poisson
equation and the trace inequality (cf. \cite{Gilbarg_Trudinger01})  that
\begin{align*}
\E(\norm{u}_{H^2(D)}^2) &\leq C\Bigl( \E(\norm{k^2u}_{L^2(D)}^2)
+\E(\norm{f}_{L^2(D)}^2) +\E(\norm{g}_{H^{\frac12}(\p D)}^2) \\
&\qquad \qquad + \E(\norm{ku}_{H^{\frac12}(\p D)}^2
+ \E(\norm{u}_{L^2(D)}^2) \Bigr)\\
&\leq  C \E\Bigl( k^4 \norm{u}_{L^2(D)}^2 + \norm{f}_{L^2(D)}^2
+\norm{g}_{H^{\frac12}(\p D)}^2 \Bigr) \\
&\qquad \qquad + C \E\Bigl( k^2 \norm{\nab u}_{L^2(D)}^2 + \norm{u}_{L^2(D)}^2 \Bigr) \\
&\leq C\Bigl(k+\frac{1}{k^2}\Bigr)^2 \E\Bigl( \norm{f}_{L^2(D)}^2 
+ \norm{g}_{H^{\frac12}(\partial D)}^2 \Bigr).
\end{align*}
Hence \eqref{eq2.9b} holds. The proof is complete.
\end{proof}

\begin{remark}
By the definition of $\gamma_0$, we see that $\gamma_0 = O\bigl(\frac{1}{kR}\bigr)$.
In practice, this is not a restrictive condition because $R$ is often taken to be proportional
to the wave length. Hence, $kR=O(1)$.
\end{remark}

As a non-trivial byproduct, the above stability estimates can be
used conveniently to establish the existence and uniqueness of
solutions to problem \eqref{eq2.1}--\eqref{eq2.2} as defined
in Definition \ref{def2.1}.

\begin{theorem}\label{thm2.2}
Let $f\in L^2(\Ome,L^2(D))$ and $g \in L^2(\Ome, L^2(\partial D)$. For each fixed pair
of positive number $k$ and $\veps$ such that $\veps(2 + \veps) < \gamma_0$, there exists a unique solution
$u\in L^2(\Ome, H^1_+(D)\cap V)$ to problem \eqref{eq2.1}--\eqref{eq2.2}.
\end{theorem}

\begin{proof}
The proof is based on the well known Fredholm Alternative Principle 
(cf. \cite{Gilbarg_Trudinger01}). First, it is easy to check that the sesquilinear 
form on the right-hand side of \eqref{eq2.1} satisfies a G\"arding's inequality on the
space $L^2(\Ome, H^1(D))$. Second, to apply the Fredholm Alternative Principle
we need to prove that solutions to the adjoint problem of \eqref{eq2.1}--\eqref{eq2.2}
is unique.  It is easy to verify that the adjoint problem
is associated with the sesquilinear form
\[
\widehat{a}(w,v) := \bigl(\nab w, \nab v\bigr)_D -k^2 \bigl( \alpha^2 w, v \bigr)_D
-\i k\Langle \alpha w, v\Rangle_{\p D},
\]
which differs from $a(\cdot,\cdot)$ only in the sign of the last term.
As a result, all the stability estimates for problem \eqref{eq2.1}--\eqref{eq2.2}
still hold for its adjoint problem. Since the adjoint problem is a linear
problem (so is problem \eqref{eq2.1}--\eqref{eq2.2}), the stability estimates
immediately infers the uniqueness. Finally, the Fredholm Alternative Principle
then implies that problem \eqref{eq2.1}--\eqref{eq2.2} has a
unique solution $u\in L^2(\Ome, H^1(D))$.  The proof is complete.
\end{proof}

\begin{remark}
The uniqueness of the adjoint problem can also be proved using the
classical unique continuation argument (cf. \cite{Leis86}).
\end{remark}

\section{Multi-modes representation of the solution and its finite modes approximations}\label{sec-3}

The first goal of this section is to develop a multi-modes representation
for the solution to problem \eqref{eq1.1}--\eqref{eq1.2} in terms of powers
of the parameter $\veps$. We first postulate such a representation and
then prove its validity by establishing some energy estimates
for all the mode functions. The second goal of this section is to
establish an error estimate for finite modes approximations
of the solution. Both the multi-modes representation and its finite
modes approximations play a pivotal role in our overall solution
procedure for solving problem \eqref{eq1.1}--\eqref{eq1.2} as
they provide the theoretical foundation for the solution procedure.
Throughout this section, we use $u^{\veps}$ to denote the solution
to problem \eqref{eq1.1}--\eqref{eq1.2} which is proved in Theorem \ref{thm2.2}. 

We start by postulating that the solution $u^\veps$ has the following
multi-modes expansion:
\begin{equation}\label{eq3.1}
u^{\veps} = \sum_{n=0}^\infty \veps^n u_n,
\end{equation}
whose validity will be justified later. Without loss of the generality,
we assume that $k\ge1$ and $D\subset B_1(0)$.  Otherwise, the problem can
be rescaled to this regime by a suitable change of variable. We note that
the normalization $D\subset B_1(0)$ implies that $R=1$.

Substituting the above expansion into the Helmholtz equation
\eqref{eq1.1} and matching the coefficients of $\veps^n$ order terms
for $n=0,1,2,\cdots$, we obtain
\begin{align}\label{eq3.2}
u_{-1} &:\equiv 0,\\
-\Delta u_0- k^2 u_0 &= f, \label{eq3.3} \\
-\Delta u_{n}- k^2 u_{n} &= 2k^2\eta u_{n-1} +k^2\eta^2u_{n-2}
\qquad \mbox{for } n\geq 1. \label{eq3.4}
\end{align}
Similarly, the boundary condition \eqref{eq1.2} translates to each
mode function $u_n$ as follows:
\begin{equation}\label{eq3.5}
\p_\nu u_n + \i ku_n = - \i k \eta u_{n-1} \qquad\mbox{for } n\geq 0.
\end{equation}

A remarkable feature of the above multi-modes expansion is that
all the mode functions satisfy the same type (nearly deterministic)
Helmholtz equation and the same boundary condition. The only
difference is that the Helmholtz equations have different
right-hand side source terms (all of them except one are random
variables), and each pair of consecutive mode functions
supply the source term for the Helmholtz equation satisfied by
the next mode function. This remarkable feature will be crucially
utilized in Section \ref{sec-5} to construct our overall numerical
methodology for solving problem \eqref{eq1.1}--\eqref{eq1.2}.

Next, we address the existence and uniqueness of each mode function $u_n$.

\begin{theorem}\label{thm3.1}
Let $f\in L^2(\Ome, L^2(D))$. Then for each $n\geq 0$, there
exists a unique solution $u_n \in L^2(\Ome, H^1(D))$ (understood in
the sense of Definition \ref{def2.1}) to problem \eqref{eq3.3},\eqref{eq3.5} for $n=0$
and problem \eqref{eq3.4},\eqref{eq3.5} for $n\geq 1$. Moreover, 
for $n\geq 0$, $u_n$ satisfies
\begin{align}\label{eq3.6}
\E\Bigl( \norm{u_n}_{L^2(D)}^2 +\norm{u_n}_{L^2(\p D)}^2 
&+ c_0\norm{\nab u_n}_{L^2(\p D)}^2 \Bigr) \\
&\leq \Bigl( \frac{1}{k} +\frac{1}{k^2} \Bigr)^2 C(n,k) \E(\norm{f}_{L^2(D)}^2), \nonumber\\
\E \bigl(\norm{u_n}_{H^1(D)}^2\bigr) &\leq  \Bigl(1+\frac{1}{k^2} \Bigr)^2
C(n,k) \E(\norm{f}_{L^2(D)}^2),  \label{eq3.7}
\end{align}
where
\begin{equation}\label{eq3.7a}
C(0,k):=C_0,\qquad C(n,k):= 4^{2n-1}C_0^{n+1} (1+k)^{2n} \quad\mbox{for } n\geq 1.
\end{equation}
Moreover, if $u_n \in L^2(\Ome, H^2(D))$, there also holds
\begin{align}\label{eq3.6a}
\E(\norm{u_n}_{H^2(D)}^2)
\leq \frac{1}{\overline{c}_0} \Bigl (k+\frac{1}{k^2} \Bigr)^2 C(n,k) 
\E(\norm{f}_{L^2(D)}^2), 
\end{align}
where $\overline{c}_0:=\min\{1,kc_0\}$.

\end{theorem}

\begin{proof}
For each $n\geq 0$, the PDE problem associated with $u_n$ is the same type
Helmholtz problem as the original problem \eqref{eq1.1}--\eqref{eq1.2}
(with $\veps=0$ in the left-hand side of the PDE). Hence, all a priori estimates
of Theorem \ref{thm2.1} hold for each $u_n$ (with its respective right-hand source
side function).  First, we have
\begin{align}\label{eq3.8}
\E\Bigl( \norm{u_0}_{L^2(D)}^2 +\norm{u_0}_{L^2(\p D)}^2 
&+ c_0\norm{\nab u_0}_{L^2(\p D)}^2 \Bigr) \\
&\leq C_0\Bigl( \frac{1}{k} +\frac{1}{k^2} \Bigr)^2 \E(\norm{f}_{L^2(D)}^2),\nonumber\\
\E(\norm{u_0}_{H^1(D)}^2) &\leq C_0 \Bigl(1+\frac{1}{k^2} \Bigr)^2
\E(\norm{f}_{L^2(D)}^2). \label{eq3.9}
\end{align}
Thus, \eqref{eq3.6} and \eqref{eq3.7} hold for $n=0$.

Next, we use induction to prove that \eqref{eq3.6} and \eqref{eq3.7}  
hold for all $n> 0$.  Assume that \eqref{eq3.6} and \eqref{eq3.7} hold for 
all $0\leq n\leq \ell-1$, then
\begin{align*}
&\E\Bigl( \norm{u_\ell}_{L^2(D)}^2 + \norm{u_\ell}_{L^2(\p D)}^2 
+ c_0\norm{\nabla u_\ell}_{L^2(\p D)}^2 \Bigr) \\
&\,
\leq 2C_0\Bigl( \frac{1}{k}+\frac{1}{k^2} \Bigr)^2
\E\Bigl( \norm{2k^2\eta u_{\ell-1} }_{L^2(D)}^2
+\overline{\delta}_{1\ell} \norm{k^2\eta^2 u_{\ell-2}}_{L^2(D)}^2 
+\norm{k \eta u_{\ell-1}}_{L^2(\p D)}^2 \Bigr) \\
&\,
\leq 2C_0\Bigl( \frac{1}{k}+\frac{1}{k^2} \Bigr)^2
(1+k)^2 \Bigl( 4C(\ell-1,k) +C(\ell-2,k) \Bigr) \E(\norm{f}_{L^2(D)}^2)\\
&\,
\leq \Bigl( \frac{1}{k}+\frac{1}{k^2} \Bigr)^2 \, 8 C_0 (1+k)^2 C(\ell-1,k)
\left( 1+ \frac{C(\ell-2,k)}{C(\ell-1,k)} \right) \E(\norm{f}_{L^2(D)}^2)\\
&\,
\leq \Bigl( \frac{1}{k}+\frac{1}{k^2} \Bigr)^2 C(\ell, k) \E(\norm{f}_{L^2(D)}^2),
\end{align*}
where $\overline{\delta}_{1\ell}=1-\delta_{1\ell}$ and $\delta_{1\ell}$ denotes
the Kronecker delta, and we have used the fact that $k \geq 1$ and
\[
8C_0 (1+k)^2 C(\ell-1,k) \left( 1+ \frac{C(\ell-2,k)}{C(\ell-1,k)} \right)
\leq C(\ell, k).
\]
Similarly, for we have
\begin{align*}
&\E\bigl(\norm{u_\ell}_{H^1(\p D)}^2 \bigr)
\leq 2C_0\Bigl( 1+\frac{1}{k^2} \Bigr)^2
\E\Bigl( \norm{2k^2\eta u_{\ell-1} }_{L^2(D)}^2
+\overline{\delta}_{1\ell} \norm{k^2\eta^2 u_{\ell-2}}_{L^2(D)}^2 \\
& \hspace{8.6cm} + \norm{k \eta u_{\ell-1}}_{L^2(\partial D)}^2 \Bigr) \\
&\hskip 0.4in
\leq 2C_0\Bigl(1 +\frac{1}{k^2} \Bigr)^2
(1+k)^2 \Bigl( 4C(\ell-1,k)+C(\ell-2,k) \Bigr) \E\bigl(\norm{f}_{L^2(D)}^2\bigr)\\
&\hskip 0.4in
\leq \Bigl( 1+\frac{1}{k^2} \Bigr)^2 C(\ell, k) \E\bigl(\norm{f}_{L^2(D)}^2\bigr).
\end{align*}
Hence, \eqref{eq3.6} and \eqref{eq3.7} hold for $n=\ell$. So the induction
argument is complete. 

We now use \eqref{eq3.6} and the elliptic theory for Possion problems directly 
to verify estimate \eqref{eq3.6a}.
\begin{align*}
&\E\bigl(\norm{u_n}_{H^2(\p D)}^2 \bigr) \leq 2C_0\Bigl( k +\frac{1}{k^2} \Bigr)^2
\E\Bigl( \norm{2k^2\eta u_{n-1} }_{L^2(D)}^2
+\norm{k^2\eta^2 u_{n-2}}_{L^2(D)}^2 \\
& \hspace{6.6cm} + \norm{k \eta u_{n-1}}_{H^{\frac12}(\p D)}^2 \Bigr) \\
&\hskip 0.4in
\leq 2C_0\Bigl( k +\frac{1}{k^2} \Bigr)^2
k^2\E\Bigl( 4\norm{ u_{n-1} }_{L^2(D)}^2
+\norm{ u_{n-2}}_{L^2(D)}^2 \\
& \hspace{6.6cm} + \frac{c_0}{kc_0}\norm{\nabla u_{n-1}}_{L^2(\p D)}^2 \Bigr) \\
&\hskip 0.4in
\leq \frac{2}{\overline{c}_0} C_0\Bigl( k +\frac{1}{k^2} \Bigr)^2
(1+k)^2 \Bigl( 4C(n-1,k)+C(n-2,k) \Bigr) \E\bigl(\norm{f}_{L^2(D)}^2\bigr)\\
&\hskip 0.4in
\leq \frac{1}{\overline{c}_0} \Bigl( k+\frac{1}{k^2} \Bigr)^2 C(n, k) 
\E\bigl(\norm{f}_{L^2(D)}^2\bigr).
\end{align*}
Hence, \eqref{eq3.6a} holds for all $n\geq 0$.  

With a priori estimates \eqref{eq3.6} and \eqref{eq3.7} in hand,
the proof of existence and uniqueness of each $u_n$ follows verbatim the proof
of Theorem \ref{thm2.2}, which we leave to the interested reader to verify.
The proof is complete.
\end{proof}

Now we are ready to justify the multi-modes representation \eqref{eq3.1} for
the solution $u^\veps$ of problem \eqref{eq1.1}--\eqref{eq1.2}.

\begin{theorem}\label{thm3.2}
Let $\{u_n\}$ be the same as in Theorem \ref{thm3.1}. Then
\eqref{eq3.1} is valid in $L^2(\Ome, H^1(D))$ provided that
$\sigma:=4\veps C_0^{\frac12}(1+k)<1$.
\end{theorem}

\begin{proof}
The proof consists of two parts: (i) the infinite
series on the right-hand side of \eqref{eq3.1} converges in $L^2(\Ome, H^1(D))$;
(ii) the limit coincides with the solution $u^\veps$.
To prove (i), we define the partial sum
\begin{equation}\label{eq3.12}
U^\veps_N:= \sum_{n=0}^{N-1} \veps^n u_n.
\end{equation}
Then for any fixed positive integer $p$ we have
\[
U^\veps_{N+p} -U^\veps_N= \sum_{n=N}^{N+p-1} \veps^n u_n
\]
It follows from Schwarz inequality and \eqref{eq3.6} that for $j=0,1$
\begin{align*}
&\E\bigl( \|U^\veps_{N+p} -U^\veps_N\|_{H^j(D)}^2 \bigr)
\leq p \sum_{n=N}^{N+p-1} \veps^{2n} \E(\|u_n\|_{H^j(D)}^2) \\
&\hskip 0.5in
\leq p \Bigl(k^{j-1} +\frac{1}{k^2} \Bigr)^2 \E(\norm{f}_{L^2(D)}^2)
\sum_{n=N}^{N+p-1} \veps^{2n} C(n,k) \nonumber\\
&\hskip 0.5in
\leq C_0p\Bigl(k^{j-1} +\frac{1}{k^2} \Bigr)^2 \E(\norm{f}_{L^2(D)}^2)
\sum_{n=N}^{N+p-1} \sigma^{2n}  \nonumber\\
&\hskip 0.5in
\leq C_0p \Bigl(k^{j-1} +\frac{1}{k^2} \Bigr)^2 \E(\norm{f}_{L^2(D)}^2)
\cdot \frac{\sigma^{2N} \bigl(1-\sigma^{2p}\bigr)}{1-\sigma^2}.  \nonumber
\end{align*}
Thus, if $\sigma<1$ we have
\[
\lim_{N\to \infty} \E\bigl( \|U^\veps_{N+p} -U^\veps_N\|_{H^1(D)}^2 \bigr)=0.
\]
Therefore, $\{U^\veps_N\}$ is a Cauchy sequence in $L^2(\Ome, H^1(D))$.
Since $L^2(\Ome, H^1(D))$ is a Banach space, then there exists a function
$U^\veps\in L^2(\Ome, H^1(D))$ such that
\[
\lim_{N\to \infty} U^\veps_N =U^\veps \qquad\mbox{in } L^2(\Ome,H^1(D)).
\]

To show (ii),  we first notice that by the definitions of $u_n$ and $U^\veps_N$,
it is easy to check that $U^\veps_N$ satisfies
\begin{align}\label{eq3.13}
&\int_\Ome \Bigl( \bigl(\nab U^\veps_N, \nab v\bigr)_D -k^2 \bigl(\alpha^2 U^\veps_N, v \bigr)_D
+\i k\Langle \alpha U^\veps_N, v\Rangle_{\p D} \Bigr) \, dP \\
&\hskip 0.5in
= \int_\Ome (f, v)_D \, dP -k^2 \veps^N \int_\Ome \bigl( \eta (2+\veps\eta) u_{N-1}
+ \eta^2 u_{N-2},\, v \bigr)_D\, dP \nonumber \\
& \hskip 1in + \i k \veps^N \int_\Ome \langle \eta u_{N-1}, v \rangle_{\partial D} \, dP \nonumber
\end{align}
for all $v\in L^2(\Ome, H^1(D))$. Where $\alpha=1+\veps \eta$.
In other words, $U_N^\veps$ solves the following Helmholtz problem:
\begin{alignat*}{2}
-\Delta U^\veps_N - k^2 \alpha^2 U^\veps_N
&=f- k^2 \veps^N \bigl( \eta (2+\veps\eta) u_{N-1} + \eta^2 u_{N-2} \bigr)
&&\qquad\mbox{in } D,\\
\p_\nu U^\veps_N + \i k\alpha U^\veps_N &= -\i k \veps^N \eta u_{N-1} &&\qquad\mbox{on } \p D.
\end{alignat*}

By \eqref{eq3.6} and Schwarz inequality we have
\begin{align*}
&k^2 \veps^N \left| \int_\Ome \bigl( \eta (2+\veps\eta) u_{N-1}
+ \eta^2 u_{N-2},\, v \bigr)_D\, dP \right|\\
&\qquad
\leq 3k^2\veps^N \Bigl( \bigl(\E(\|u_{N-1}\|_{L^2(D)}^2)\bigr)^{\frac12}
+ \bigl(\E(\|u_{N-2}\|_{L^2(D)}^2)\bigr)^{\frac12} \Bigr)
\bigl(\E(\|v\|_{L^2(D)}^2)\bigr)^{\frac12} \nonumber\\
&\qquad
\leq 6k^2\veps^N \Bigl(\frac{1}{k} +\frac{1}{k^2} \Bigr)
C(N-1,k)^{\frac12} \bigl(\E(\|f\|_{L^2(D)}^2)\bigr)^{\frac12}
\bigl(\E(\|v\|_{L^2(D)}^2)\bigr)^{\frac12} \nonumber \\
&\qquad
\leq 3\veps (k+1)C_0^{\frac12} \sigma^{N-1} \bigl(\E(\|f\|_{L^2(D)}^2)\bigr)^{\frac12}
\bigl(\E(\|v\|_{L^2(D)}^2)\bigr)^{\frac12} \nonumber \\
&\qquad
\longrightarrow 0 \quad\mbox{as } N\to \infty\quad\mbox{provided that } \sigma<1.
\nonumber
\end{align*}
Similarly we get
\begin{align*}
	&k \veps^N \left| \int_{\Ome} \langle \eta u_{N-1},v \rangle_{\partial D} \, dP \right| \\
	& \qquad \leq k \veps^N \bigl( \E(\| u_{N-1} \|^2_{L^2(\partial D)})\bigr)^{\frac12}\bigl( \E(\| v \|^2_{L^2(\partial D)})\bigr)^{\frac12} \\
	& \qquad \leq k \veps^N \Bigl( \frac{1}{k} + \frac{1}{k^2} \Bigr) C(N-1,k)\bigl( \E(\| f \|^2_{L^2(D)})\bigr)^{\frac12}\bigl( \E(\| v \|^2_{L^2(\partial D)})\bigr)^{\frac12} \\
	& \qquad \leq \frac{\veps}{2} \Bigl(1 + \frac{1}{k} \Bigr) C_0^{\frac12} \sigma^{N-1}\bigl( \E(\| f \|^2_{L^2(D)})\bigr)^{\frac12}\bigl( \E(\| v \|^2_{L^2(\partial D)})\bigr)^{\frac12} \\
	&\qquad \longrightarrow 0 \quad\mbox{as } N\to \infty\quad\mbox{provided that } \sigma<1.
\end{align*}
Setting $N\to \infty$ in \eqref{eq3.13} immediately yields
\begin{align}\label{eq3.15}
\int_\Ome \Bigl( \bigl(\nab U^\veps, \nab v\bigr)_D &-k^2 \bigl(\alpha^2 U^\veps, v \bigr)_D
+\i k\Langle \alpha U^\veps, v\Rangle_{\p D} \Bigr) \, dP \\
&\quad
= \int_\Ome (f, v)_D \, dP \qquad\forall v\in L^2(\Ome, H^1(D)). \nonumber
\end{align}
Thus, $U^\veps$ is a solution to problem \eqref{eq1.1}--\eqref{eq1.2}.
By the uniqueness of the solution, we conclude that $U^\veps=u^\veps$.
Therefore, \eqref{eq3.1} holds
in $L^2(\Ome,H^1(D))$. The proof is complete.
\end{proof}

The above proof also infers an upper bound for the error $u^\veps- U^\veps_N$ as
stated in the next theorem.

\begin{theorem}\label{thm3.3}
Let $U^\veps_N$ be the same as above and $u^{\veps}$ denote the solution to
problem \eqref{eq1.1}--\eqref{eq1.2} and $\sigma:=4\veps C_0^{\frac12}(1+k)$.
Then there holds for $\veps(2\veps+1)<\gamma_0$
\begin{equation}\label{eq3.16}
\E(\norm{u^\veps -U^\veps_N}_{H^j(D)}^2)
\leq \frac{9 C_0\sigma^{2N}}{32(1+k)^2} \Bigl(k^j +\frac{1}{k} \Bigr)^4
\E(\|f\|_{L^2(D)}^2), \quad j=0,1,
\end{equation}
provided that $\sigma<1$. Where $C$ is a positive constant independent of $k$ and $\veps$.
\end{theorem}

\begin{proof}
Let $E^{\veps}_N:= u^\veps -U^\veps_N$, subtracting \eqref{eq3.13} from \eqref{eq3.15} we get
\begin{align}\label{eq3.17}
&\int_\Ome \Bigl( \bigl(\nab E^\veps_N, \nab v\bigr)_D -k^2 \bigl(\alpha^2 E^\veps_N, v \bigr)_D
+\i k\Langle \alpha E^\veps_N, v\Rangle_{\p D} \Bigr) \, dP \\
&\qquad
=k^2 \veps^N \int_\Ome \bigl( \eta (2+\veps\eta) u_{N-1}
+ \eta^2 u_{N-2},\, v \bigr)_D\, dP \nonumber \\
&\qquad \qquad - \i k \veps^N \int_\Ome \langle \eta u_{N-1}, v \rangle_{\partial D} \, dP \qquad\forall v\in L^2(\Ome, H^1(D)). \nonumber
\end{align}
In other words, $E^\veps_N$ solves the following Helmholtz problem:
\begin{alignat*}{2}
-\Delta E^\veps_N - k^2 \alpha^2 E^\veps_N
&= k^2 \veps^N \bigl( \eta (2+\veps\eta) u_{N-1} + \eta^2 u_{N-2} \bigr)
&&\qquad\mbox{in } D,\\
\p_\nu E^\veps_N + \i k\alpha E^\veps_N &= - \i k \veps^N \eta u_{N-1} &&\qquad\mbox{on } \p D.
\end{alignat*}

By Theorem \ref{thm2.1} and \eqref{eq3.6} we obtain for $j=0,1$
\begin{align*}
\E(\|E^\veps_N\|_{H^j(D)}^2)
&\leq 18 C_0 \Bigl(k^{j-1} +\frac{1}{k^2} \Bigr)^2\,
\Bigl[k^4\veps^{2N} \Bigl(\E(\|u_{N-1}\|_{L^2(D)}^2) + \E(\|u_{N-2}\|_{L^2(D)}^2) \Bigr)\\
& \hspace{4cm} + k^2 \veps^{2N} \E (\| u_{N-1} \|^2_{L^2(\partial D)}) \Bigr] \\
&\leq 18C_0 k^4\veps^{2N} \Bigl(k^{j-1} +\frac{1}{k^2} \Bigr)^4
C(N-1,k)\, \E(\|f\|_{L^2(D)}^2) \nonumber  \\
&\leq \frac{18C_0 \sigma^{2N}}{64(1+k)^2} \Bigl(k^j +\frac{1}{k} \Bigr)^4
\E(\|f\|_{L^2(D)}^2). \nonumber
\end{align*}
The proof is complete.
\end{proof}

\section{Monte Carlo discontinuous Galerkin approximations of the mode functions $\mathbf{\{u_n\}}$} \label{sec-4}

In the previous section, we present a multi-modes representation of the
solution $u^\veps$ and a convergence rate estimate for its finite
approximations. These results will serve as the theoretical foundation
for our overall numerical methodology for approximating the solution $u^\veps$
of problem \eqref{eq1.1}--\eqref{eq1.2}.  To compute $\E(u^\veps)$ following
this approach, we need to compute the expectations $\{\E(u_n)\}$ of the first
$N$ mode functions $\{u_n\}_{n=0}^{N-1}$. This requires the construction of
an accurate and robust numerical (discretization) method to compute the expectations 
of the solutions to the ``nearly" deterministic Helmholtz problems
\eqref{eq3.3},\eqref{eq3.5} and \eqref{eq3.4},\eqref{eq3.5} satisfied
by the mode functions $\{u_n\}$. The construction of such a numerical 
method is exactly our focus in this section.  We note that due to the 
multiplicative structure of the right-hand side of \eqref{eq3.4}, $\E(u_n)$ 
can not be computed directly for $n\geq 1$. On the other hand,
$\E(u_0)$ can be computed directly because it satisfies
the deterministic Helmholtz equation with the source term $\E(f)$
and homogeneous boundary condition.

The goal of this section is to develop some {\em Monte Carlo interior
penalty discontinuous Galerkin} (MCIP-DG) methods for the above mentioned
Helmholtz problems. Our MCIP-DG methods are the
direct generalizations of the deterministic IP-DG methods
proposed in \cite{Feng_Wu09, Feng_Wu11} for the related deterministic
Helmholtz problems. It should be noted that although various numerical methods
(such as finite difference, finite element and spectral methods)
can be used for the job, the IP-DG methods to be presented below are
the only general purpose discretization methods which are unconditionally stable
(i.e., stable without mesh constraint) and optimally convergent.
This is indeed the primary reason why we choose the IP-DG methods
as our spatial discretization methods.

\subsection{DG notations}\label{sec-4.1}
Let $\cT_h$ be a quasi-uniform partition of $D$ such that
$\overline{D}=\bigcup_{K\in\cT_h} \overline{K}$. Let $h_K$ denote
the diameter of $K\in \cT_h$ and $h:=\mbox{max}\{h_K; K\in\cT_h\}$.
$H^s(\cT_h)$ denotes the standard broken Sobolev space and $V^h_r$ denotes
the DG finite element space which are defined as
\[
H^s(\cT_h):=\prod_{K\in\cT_h} H^{s}(K), \qquad
V^h_r:=\prod_{K\in\cT_h} P_r(K),
\]
where $P_r(K)$ is the set of all polynomials whose degrees do not
exceed a given positive integer $r$.  Let $\cE^I$ denote the set of all
interior faces/edges of $\cT_h$, $\cE^B$ denote the set of all boundary
faces/edges of $\cT_h$, and $\cE:=\cE^I\cup \cE^B$. The $L^2$-inner product
for piecewise functions over the mesh $\cT_h$ is naturally defined by
\[
(v,w)_{\cT_h}:= \sum_{K\in \cT_h} \int_{K} v w\, dx,
\]
and for any set $\cS_h \subset \cE$, the $L^2$-inner product
over $\cS_h$ is defined by
\begin{align*}
\Langle v,w\Rangle_{\cS_h} :=\sum_{e\in \cS_h} \int_e v w\, dS.
\end{align*}

Let $K, K'\in \cT_h$ and $e=\partial K\cap \partial K'$ and assume
global labeling number of $K$ is smaller than that of $K'$.
We choose $n_e:=n_K|_e=-n_{K'}|_e$ as the unit normal on $e$ outward to $K$ and
define the following standard jump and average notations across the face/edge $e$:
\begin{alignat*}{4}
[v] &:= v|_K-v|_{K'}
\quad &&\mbox{on } e\in \cE^I,\qquad
&&[v] :=v\quad
&&\mbox{on } e\in \cE^B,\\
\{v\} &:=\frac12\bigl( v|_K +v|_{K'} \bigr) \quad
&&\mbox{on } e\in \cE^I,\qquad
&&\{v\}:=v\quad
&&\mbox{on } e\in \cE^B
\end{alignat*}
for $v\in V^h_r$. We also define the following semi-norms on $H^s(\cT_h)$:
\begin{align*}
\abs{v}_{1,h,D} &:=\norm{\nab v}_{L^2(\cT_h)},\\
\norm{v}_{1,h,D} &:=\left(\abs{v}_{1,h,D}^2
+\sum_{e\in\cE_h^I} \left(\frac{\gamma_{0,e}\,r}{h_e}\norm{[v]}_{L^2(e)}^2
+\sum_{\ell=1}^{d-1} \frac{\beta_{1,e}r}{h_e}\norm{[\p_{\tau_e^\ell} v]}_{L^2(e)}^2
\right) \right. \\
&\hskip 0.74in \left.
+ \sum_{j=1}^r\sum_{e\in\cE_h^I} \gamma_{j,e} \Bigl(\frac{h_e}{r}\Bigr)^{2j-1}
\norm{[\p^j_{n_e} v]}_{L^2(e)}^2 \right)^{\frac12},  \\
|||v|||_{1,h,D} &:=\left(\norm{v}_{1,h,D}^2 +\sum_{e\in\cE_h^I}
\frac{h_e}{\gamma_{0,e} r} \norm{ \{\p_{n_e} v\} }_{L^2(e)}^2 \right)^{\frac12}.
\end{align*}

\subsection{IP-DG method for deterministic Helmholtz problem}\label{sec-4.2}
In this subsection we consider following deterministic Helmholtz problem and
its IP-DG approximations proposed in \cite{Feng_Wu09,Feng_Wu11}.
\begin{alignat}{2}\label{eq4.1}
-\Delta \Phi_0 - k^2 \Phi_0 &=F_0 &&\qquad\mbox{in } D,\\
\p_\nu \Phi_0 +\i k \Phi_0 &=G_0 &&\qquad\mbox{on }\p D. \label{eq4.2}
\end{alignat}
We note that $\Phi_0=\E(u_0)$ satisfies the above equations with $F_0=\E(f)$ 
and $G_0=0$.  As an interesting byproduct, all the results to be presented in 
this subsection apply to $\E(u_0)$.

The IP-DG weak formulation for \eqref{eq4.1}--\eqref{eq4.2} is defined by
(cf. \cite{Feng_Wu09, Feng_Wu11}) seeking 
$\Phi_0\in H^1(D)\cap H^{r+1}_{\mbox{\tiny loc}}(D)$ such that
\begin{equation}\label{eq4.3}
a_h(\Phi_0, \psi) = (F_0, \psi)_D +\langle G_0, \psi \rangle_{\p D}
\quad\qquad\forall \psi\in H^1(D)\cap H^{r+1}(\cT_h),
\end{equation}
where
\begin{align}\label{eq4.4}
a_h(\phi,\psi) &:= b_h(\phi,\psi) -k^2(\phi,\psi)_{\cT_h} +\i k\langle \phi,\psi\rangle_{\cE^B_h}
+ \i \Bigl( L_1(\phi,\psi) +\sum_{j=0}^r J_j(\phi,\psi) \Bigr), \\
b_h(\phi,\psi) &:= (\nab \phi,\nab \psi)_{\cT_h}
- \Bigl( \Langle \{\p_n \phi\}, [\psi] \Rangle_{\cE^I_h}
+ \Langle [\phi], \{\p_n \psi\} \Rangle_{\cE_h^I}  \Bigr),
\nonumber \\
L_1(\phi,\psi) &:=\sum_{e\in\cE_h^I}\sum_{\ell=1}^{d-1} \beta_{1,e}{h_e}^{-1}
\Langle [\p_{\tau^\ell} \phi], [\p_{\tau^\ell} \psi] \Rangle_e, \nonumber \\
J_j(\phi,\psi) &:=\sum_{e\in\cE_h^I} \gamma_{j,e} h_e^{2j-1}
\Langle [\p_n^j \phi],[\p_n^j \psi] \Rangle_e, \qquad j=0,1,\cdots,r. \nonumber
\end{align}
$\{\beta_{1,e}\}$ and $\{\gamma_{j,e}\}$ are piecewise constant nonnegative functions 
defined on $\cE_h^I$. $\{\tau^\ell\}_{\ell=1}^{d-1}$ denotes an orthonormal basis of the 
edge and $\p_{\tau^\ell}$ denotes the tangential derivative in the direction of $\tau^\ell$.

\begin{remark}\label{rem-4.1}
$L_1$ and $\{J_j\}$ terms are called interior penalty terms, $\{\beta_{1,e}\}$ and
$\{\gamma_{j,e}\}$ are called penalty parameters. The two distinct features of the DG
sesquilinear form $a_h(\cdot,\cdot)$ are: (i) it penalizes not only the jumps of the
function values but also penalizes the jumps of the tangential derivatives as well the
jumps of all normal derivatives up to $r$th order; (ii) the penalty parameters are pure
imaginary numbers with nonnegative imaginary parts.
\end{remark}

Following \cite{Feng_Wu09, Feng_Wu11} and based on the DG weak formulation \eqref{eq4.3},
our IP-DG method for problem \eqref{eq4.1}--\eqref{eq4.2} is defined by seeking
$\Phi_0^h\in V^h_r$ such that
\begin{equation}\label{eq4.5}
a_h(\Phi_0^h, \psi^h) = (F_0, \psi^h)_D +\langle G_0, \psi^h \rangle_{\p D}
\quad\qquad\forall \psi^h\in V^h_r.
\end{equation}

For the above IP-DG method, it was proved in \cite{Feng_Wu09, Feng_Wu11}
that the method is unconditionally stable and its solutions satisfy some
wave-number-explicit stability estimates. Its solutions also satisfy  
optimal order (in $h$) error estimates, which are described below.

\begin{theorem}\label{thm4.1}
Let $\Phi_0^h\in V^h_r$ be a solution to scheme \eqref{eq4.5}, then there hold\\

(i) For all $h,k>0$, there exists a positive constant $\hat{C}_0$ independent
of $\veps$ and $h$ such that

\begin{align}\label{eq4.6}
\|\Phi_0^h\|_{L^2(D)} +\frac{1}{k} \norm{\Phi_0^h}_{1,h,D}
+\|\Phi_0^h\|_{L^2(\p D)} 
\leq \hat{C}_0 C_s\, \hM(F_0,G_0),
\end{align}
where
\begin{align}\label{eq4.7}
&C_s:=\frac{d-2}k+\frac{1}{k^2}
+\frac{1}{k^2}\max_{e\in\cE_h^I} \Bigl(\,\frac{r\,k^2h_e^2+r^5}{\gamma_{0,e}\,h_e^2}
+\frac{r}{h_e} \max_{0\le j\le r-1}\sqrt{\frac{\gamma_{j,e}}{\gamma_{j+1,e}}} \\
&\hskip 1.1in
+\frac{r^2}{h_e} +\frac{r^3}{h_e^2}\sqrt{\frac{\beta_{1,e}}{\gamma_{1,e}}}\,\Bigr),
\nonumber\\
&\hM(F_0,G_0):= \|F_0\|_{L^2(D)} + \|G_0\|_{L^2(\p D)}. \label{eq4.7a}
\end{align}

(ii) If $k^3h^2r^{-2} =O(1)$, then there exists a positive constant $\hat{C}_0$
independent of $k$ and $h$ such that
\begin{equation}\label{eq4.8}
\|\Phi_0^h\|_{L^2(D)} + \|\Phi_0^h\|_{L^2(\p D)}
+ \frac{1}{k} \|\Phi_0^h\|_{1,h,D}
\leq \hat{C}_0 \Bigl( \frac{1}{k} +\frac{1}{k^2} \Bigr) \hM(F_0,G_0).
\end{equation}

\end{theorem}

An immediate consequence of \eqref{eq4.6} is the following unconditional solvability
and uniqueness result.

\begin{corollary}\label{cor4.1}
There exists a unique solution to scheme \eqref{eq4.5} for all $k, h>0$.
\end{corollary}

\begin{theorem}\label{thm4.2}
Let $\Phi_0^h\in V^h$ solve \eqref{eq4.5},  $\Phi_0\in H^s(\Ome)$ be the
solution of \eqref{eq4.1}--\eqref{eq4.2}, and $\mu=\min\{r+1,s\}$.
Suppose $\gamma_{j,e}, \beta_{1,e}>0$. Let $\gamma_j=\max_{e\in \cE^I}
\gamma_{j,e}$ and $\lambda=1+\frac{1}{\gamma_0}$. \\

(i) For all $h,k>0$, there exists a positive constant $\tilde{C}_0$ independent
of $\veps$ and $h$ such that
\begin{align}\label{eq4.9}
&\|\Phi_0-\Phi_0^h\|_{1,h,D} \leq \tilde{C}_0 \Bigl( C_r + \frac{k^3 h}{r}\,
C_s \hat{C}_r \Bigr)\,\frac{h^{\mu-1}}{r^{s-1}}\|\Phi_0\|_{H^s(D)},  \\
&\|\Phi_0-\Phi_0^h\|_{L^2(D)} +\|\Phi_0-\Phi_0^h\|_{L^2(\p D)} \leq \tilde{C}_0
\hat{C}_r\, \Bigl(1 +  k^2C_s \Bigr)\,\frac{h^\mu}{r^s}\|\Phi_0\|_{H^s(D)}, \label{eq4.10}
\end{align}
where
\begin{align*}
C_r &:= \lambda \Big(1+\frac{r}{\gamma_0}+\sum_{j=1}^r r^{2j-1}\gamma_j
+\frac{kh}{\lambda r}\Big)^{\frac12}, \nonumber\\
\hat{C}_r &:=\Big(1+\frac{r}{\gamma_0}+ r\,\gamma_1
+\sum_{j=2}^r r^{2j-2}\gamma_j+\frac{kh}{\lambda r}\Big)^{\frac12}\,C_r. \nonumber
\end{align*}

(ii) If $k^3h^2r^{-2} =O(1)$, then there exists a positive constant $\tilde{C}_0$
independent of $k$ and $h$ such that
\begin{align}\label{eq4.11}
&\norm{\Phi_0-\Phi_0^h}_{1,h,D}
\leq \frac{ \tilde{C}_0(r+k^2 h)h^{\mu-1}}{r^s}\,\norm{\Phi_0}_{H^{s}(D)},\\
&\norm{\Phi_0-\Phi_0^h}_{L^2(D)}
+\norm{\Phi_0-\Phi_0^h}_{L^2(\p D)}
\leq \frac{\tilde{C}_0 kh^\mu}{r^s}\norm{\Phi_0}_{H^{s}(D)}. \label{eq4.12}
\end{align}

\end{theorem}

\begin{remark}\label{rem4.2}
It was proved in \cite{Cummings_Feng06} (also by Theorem \ref{thm2.1} with $\veps=0$) that
\[
\|\Phi_0\|_{H^s(D)} \leq \tilde{C}_0 \Bigl( k^{s-1}+ \frac{1}{k} \Bigr) \hM(F_0,G_0),
\qquad s=0,1,2.
\]
It is expected that the following higher order norm estimates also hold
(cf. \cite{Feng_Wu09} for an explanation):
\begin{equation}\label{eq4.13}
\|\Phi_0\|_{H^s(D)} \leq \tilde{C}_0 \Bigl( k^{s-1} + \frac{1}{k} \Bigr) 
\Bigl( \|F_0\|_{H^{s-2}(D)} +\|G_0\|_{H^{s-\frac52}(\p D))} \Bigr),
\qquad s\geq 3
\end{equation}
provided that $F_0$, $G_0$ and $D$ are sufficiently smooth. In such a case,
$\norm{\Phi_0}_{H^{s}(D)}$ in \eqref{eq4.9}--\eqref{eq4.12} can be replaced by
the above bound so explicit constants can be obtained in these estimates.
\end{remark}

\subsection{MCIP-DG method for approximating $\mathbf{\E(u_n)}$ for $\mathbf{n\geq 0}$}
\label{sec-4.3}

We recall that each mode function $u_n$ satisfies the following Helmholtz problem:
\begin{alignat}{2}\label{eq4.20}
-\Delta u_n - k^2  u_n &= S_n  &&\qquad \mbox{in } D,  \\
\p_\nu u_n +\i k u_n &=Q_n &&\qquad \mbox{on } \partial D, \label{eq4.21}
\end{alignat}
where
\[
u_{-1}:=0,\, S_0:=f, \, Q_0:=0,\, S_n:= 2k^2\eta u_{n-1} +k^2\eta^2u_{n-2},
\, Q_n:=-\i k\eta u_{n-1},\, n\geq 1.
\]
Clearly, $S_n(x,\cdot)$ and $Q_n$ are random variables for $a.e.\, x\in D$,
$S_n\in L^2(\Ome, L^2(D))$ and $Q_n\in L^2(\Ome, L^2(\p D))$. We remark again that 
due to its multiplicative structure $\E(S_n)$ and $\E(Q_n)$ can not be computed 
directly for $n\geq 1$. Otherwise, \eqref{eq4.20} and \eqref{eq4.21} would be 
easily converted into deterministic equations for $\E(u_n)$,
as we did early for $\E(u_0)$. In other words, \eqref{eq4.20}--\eqref{eq4.21} is
a genuine random PDE problem. On the other hand, since all the coefficients
of the equations are constants, then the problem is nearly deterministic.
Such a remarkable property will be fully exploited in our overall numerical
methodology which will be described in the next section.

Several numerical methodologies are well known in the literature for discretizing
random PDEs, Monte Carlo Galerkin and stochastic Galerkin
(or polynomial chaos)
methods and stochastic collocation methods are three of well-known methods (cf.
\cite{Babuska_Tempone_Zouraris04,Babuska_Nobile_Tempone10} and the
references therein). Due to the nearly deterministic structure of
\eqref{eq4.20}--\eqref{eq4.21}, we propose to discretize it using
the Monte Carlo IP-DG approach which combines the classical Monte Carlo
method for stochastic variable and the IP-DG method, which is presented in the
proceeding subsection, for the spatial variable.

Following the standard formulation of the Monte Carlo method
(cf. \cite{Babuska_Tempone_Zouraris04}),  let $M$ be a (large) positive
integer which will be used to denote the number of realizations and
$V^h_r$ be the DG space defined in Section \ref{sec-4.1}.  For each
$j=1,2,\cdots, M$, we sample i.i.d. realizations of the source term
$f(\omega_j,\cdot)$ and random medium coefficient $\eta(\omega_j,\cdot)$, 
and recursively find corresponding approximation
$u_n^h(\omega_j,\cdot)\in V^h_r$ such that

\begin{align}\label{eq4.22}
a_h\bigl( u_n^h(\omega_j,\cdot), \psi^h\bigr) =\bigl( S^h_n(\omega_j,\cdot), \psi^h \bigr)_D
+ \langle Q^h_n(\omega_j,\cdot), \psi^h \rangle_{\p D} \qquad\forall \psi^h\in V^h_r
\end{align}
for $n=0,1,2,\cdots, N-1$. Where
\begin{align}\label{eq4.23}
S^h_0(\omega_j,\cdot) &:=f(\omega_j,\cdot),\quad Q^h_0:=0, \\
u^h_{-1}(\omega_j,\cdot) &:= 0,  \label{eq4.24} \\
S^h_n(\omega_j,\cdot) &:= 2k^2\eta u^h_{n-1}(\omega_j,\cdot)
+k^2\eta^2 u^h_{n-2}(\omega_j,\cdot), \qquad n=1,2,\cdots, N-1, \label{eq4.25}\\
Q^h_n(\omega_j,\cdot) &:= -\i k\eta u^h_{n-1}(\omega_j,\cdot), 
\qquad n=1,2,\cdots, N-1. \label{eq4.25a}
\end{align}
We point out that in order for $u^h_n$ to be computable, $S^h_n$ and $Q^h_n$, not $S_n$
and $Q_n$, are used on the right-hand side of \eqref{eq4.22}. This (small) perturbation
on the right-hand side will result in an additional discretization error which
must be accounted later, see Section \ref{sec-5}.

Next, we approximate $\E(u_n)$ by the following sample average
\begin{align}\label{eq4.26}
\Phi^h_n:= \frac{1}{M} \sum_{j=1}^M u_n^h(\omega_j,\cdot).
\end{align}

The following lemma is well known (cf.  \cite{Babuska_Tempone_Zouraris04, Liu_Riviere13}).

\begin{lemma}\label{lem4.1}
There hold the following estimates for $n\geq 0$
\begin{align}\label{eq4.27}
\E\bigl(\|\E(u^h_n) -\Phi^h_n\|_{L^2(D)}^2 \bigr)
&\leq \frac{1}{M} \E(\|u^h_n\|_{L^2(D)}^2),\\
\E\bigl(\|\E(u^h_n) -\Phi^h_n\|_{1,h,D}^2 \bigr)
&\leq \frac{1}{M} \E(\|u^h_n\|_{1,h,D}^2). \label{eq4.27a}
\end{align}
\end{lemma}

To bound $\E(\|u^h_n\|_{1,h,D}^2)$, we once again use the induction argument.
To avoid some technicalities, we only provide a proof for the case
when the mesh size is in pre-asymptotic regime, i.e., $k^3h^2r^{-2}=O(1)$.

\begin{lemma}\label{lem4.2}
Assume $k^3h^2r^{-2}=O(1)$. Then there hold for $n\geq 0$
\begin{align}\label{eq4.28}
&\E\bigl( \|u^h_n\|_{L^2(D)}^2 + \|u^h_n\|_{L^2(\p D)}^2 \bigr)
\leq \Bigl(\frac{1}{k}+\frac{1}{k^2}\Bigr)^2
\hat{C}(n,k)\, \E(\|f\|_{L^2(D)}^2),\\
&\E(\|u^h_n\|_{1,h,D}^2) \leq \Bigl(1+\frac{1}{k}\Bigr)^2 \hat{C}(n,k)\,
\E(\|f\|_{L^2(D)}^2), \label{eq4.29}
\end{align}
where
\begin{equation}\label{eq4.30}
\hat{C}(0,k):=\hat{C}_0^2,\quad \hat{C}(n,k):=4^{2n-1}\hat{C}_0^{2n+2}(1+k)^{2n}
\quad\mbox{for } n\geq 1.
\end{equation}
\end{lemma}

\begin{proof}
By \eqref{eq4.22} and estimate \eqref{eq4.8} we immediately get
\begin{align*}
\E\bigl( \|u^h_0\|_{L^2(D)}^2 &+ \|u^h_0\|_{L^2(\p D)}^2 \bigr) \\
&\leq \hat{C}_0^2 \Bigl( \frac{1}{k} +\frac{1}{k^2}\Bigr)^2 \E(\|S^h_0\|_{L^2(D)}^2)
\leq \hat{C}_0^2 \Bigl( \frac{1}{k} +\frac{1}{k^2}\Bigr)^2 \E(\|f\|_{L^2(D)}^2), \\
\E\bigl( \|u^h_0\|_{1,h,D}^2 \bigr) &\leq \hat{C}_0^2 \Bigl(1+\frac{1}{k}\Bigr)^2
\E(\|S^h_0\|_{L^2(D)}^2)
\leq \hat{C}_0^2  \Bigl(1+\frac{1}{k}\Bigr)^2 \E(\|f\|_{L^2(D)}^2),
\end{align*}
which verifies \eqref{eq4.28} and \eqref{eq4.29} for $n=0$.
Suppose \eqref{eq4.28} and \eqref{eq4.29} hold for all
$n=0,1,2,\cdots, \ell-1$, we now prove that they also hold for $n=\ell$.

Again, by \eqref{eq4.22} with $n=\ell-1$ and estimate \eqref{eq4.8} we have
\begin{align*}
&\E\bigl( \|u^h_\ell\|_{L^2(D)}^2 +\|u^h_\ell\|_{L^2(\p D)}^2 \bigr)
\leq\hat{C}_0^2\Bigl(\frac{1}{k}+\frac{1}{k^2}\Bigr)^2
\E\Bigl(\|S^h_\ell\|_{L^2(D)}^2 + \|Q^h_n\|_{L^2(\p D)}^2 \Bigr) \\
&\quad
\leq 2\hat{C}_0^2 \Bigl(\frac{1}{k}+\frac{1}{k^2}\Bigr)^2
 k^4 E\Bigl( 4\|u^h_{\ell-1}\|_{L^2(D)}^2 
+\|u^h_{\ell-2}\|_{L^2(D)}^2 + \frac{1}{k^2}\|u^h_{\ell-1}\|_{L^2(\p D)}^2 \Bigr) \\
&\quad
\leq 2\hat{C}_0^2 \Bigl(\frac{1}{k}+\frac{1}{k^2}\Bigr)^2
(1+k)^2 \Bigl( 4\hat{C}(\ell-1,k)
+ \hat{C}(\ell-2,k) \Bigr) \E(\|f\|_{L^2(D)}^2)\\
&\quad
\leq 8\hat{C}_0^2 \Bigl(\frac{1}{k}+\frac{1}{k^2}\Bigr)^2 (1+k)^2 \hat{C}(\ell-1,k)
\left( 1+ \frac{\hat{C}(\ell-2,k)}{4\hat{C}(\ell-1,k)} \right)
\E(\|f\|_{L^2(D)}^2)\\
&\quad
\leq \Bigl(\frac{1}{k}+\frac{1}{k^2}\Bigr)^2 \hat{C}(\ell,k) \E(\|f\|_{L^2(D)}^2),
\end{align*}
here we have used the fact that
\[
 8\hat{C}_0^2 (1+k)^2 \hat{C}(\ell-1,k)
\left( 1+ \frac{\hat{C}(\ell-2,k)}{4\hat{C}(\ell-1,k)} \right)
\leq  \hat{C}(\ell,k).
\]
Similarly, we have
\begin{align*}
\E(\|u^h_\ell\|_{1,h,D}^2)&\leq\hat{C}_0^2\Bigl(1+\frac{1}{k}\Bigr)^2
\E\Bigl( \|S^h_\ell\|_{L^2(D)}^2 + \|Q^h_n\|_{L^2(\p D)}^2 \Bigr) \\
&\leq 2\hat{C}_0^2 \Bigl(1+\frac{1}{k}\Bigr)^2
k^4 E\Bigl( 4\|u^h_{\ell-1}\|_{L^2(D)}^2
+\E(\|u^h_{\ell-2}\|_{L^2(D)}^2 + \frac{1}{k^2}\|u^h_{\ell-1}\|_{L^2(\p D)}^2 \Bigr) \\
&\leq 2\hat{C}_0^2 \Bigl(1+\frac{1}{k}\Bigr)^2 (1+k)^2 \Bigl( 4\hat{C}(\ell-1,k)
+ \hat{C}(\ell-2,k) \Bigr) \E(\|f\|_{L^2(D)}^2)\\
&\leq \Bigl(1+\frac{1}{k}\Bigr)^2 \hat{C}(\ell,k) \E(\|f\|_{L^2(D)}^2).
\end{align*}
This completes the induction argument and the proof.
\end{proof}

Combining Lemmas \ref{lem4.1} and \ref{lem4.2}, we have

\begin{theorem}\label{thm4.3}
Suppose $k^3h^2r^{-2}=O(1)$. Then there hold
\begin{align}\label{eq4.31}
\E\bigl(\|\E(u^h_n) -\Phi^h_n\|_{L^2(D)}^2 \bigr)
&\leq \frac{1}{M} \Bigl(\frac{1}{k}+\frac{1}{k^2}\Bigr)^2
\hat{C}(n,k)\, \E(\|f\|_{L^2(D)}^2),\\
\E\bigl(\|\E(u^h_n) -\Phi^h_n\|_{1,h,D}^2 \bigr)
&\leq \frac{1}{M} \Bigl(1+\frac{1}{k}\Bigr)^2
\hat{C}(n,k)\, \E(\|f\|_{L^2(D)}^2), \label{eq4.32}
\end{align}
\end{theorem}

\begin{remark}\label{rem4.1}
Estimates \eqref{eq4.31} and \eqref{eq4.32} show that for each fixed $n\geq 0$
the statistical error due to sampling is controlled by the number of realizations of $u^h_n$.
Indeed, it can be easily proved by using Markov's inequality and Borel-Cantelli lemma
that the statistical error converges to zero as $M$ tends to infinity,
see \cite[Proposition 4.1]{Babuska_Tempone_Zouraris04} and \cite[Theorem 3.2]{Liu_Riviere13}.
\end{remark}

\section{The overall numerical procedure}\label{sec-5}

\subsection{The numerical algorithm, linear solver and computational complexity}\label{sec-5.1}
We are now ready to introduce our overall numerical procedure for approximating
the solution of the original random Helmholtz problem \eqref{eq1.1}--\eqref{eq1.2}.
Our numerical procedure consists of three main ingredients. First, it is based on
the multi-modes representation \eqref{eq3.1} and its finite modes approximation
\eqref{eq3.12}. Second, it uses the classical Monte Carlo method for sampling
the probability space and for computing the expectations of the numerical
solutions. Finally, at each realization an IP-DG method is employed to
solve all the involved deterministic Helmholtz problems. The precise
description of this procedure is given by the following algorithm.

\medskip
\noindent
{\bf Main Algorithm}
\smallskip

\begin{description}
\item Inputs: $f, \eta, \veps, k, h, M, N.$
\item Set $\Psi^h_N(\cdot)=0$ (initializing).
\begin{description}
\item For $j=1,2,\cdots, M$
\item Set $S^h_0(\omega_j,\cdot)=f(\omega_j,\cdot)$.
\item Set $Q^h_0(\omega_j,\cdot)=0$.
\item Set $u^h_{-1}(\omega_j,\cdot)=0$.
\item Set $U^h_N(\omega_j,\cdot)=0$ (initializing).
\begin{description}
\item For $n=0,1,\cdots, N-1$
\item Solve for $u^h_n(\omega_j,\cdot) \in V^h_r$ such that
\[
a_h\bigl( u^h_n(\omega_j,\cdot), v_h \bigr) = \bigl(S^h_n(\omega_j,\cdot), v_h\bigr)_D + \big \langle Q^h_n(\omega_j,\cdot),v_h \big \rangle_{\partial D}
\qquad\forall v_h\in V^h_r.
\]
\item Set $U^h_N(\omega_j,\cdot)\leftarrow U^h_N(\omega_j,\cdot) +\veps^n u^h_n(\omega_j,\cdot)$.
\item Set $S^h_{n+1}(\omega_j,\cdot)=2k^2 \eta(\omega_j,\cdot) u^h_n(\omega_j,\cdot)
+ k^2 \eta(\omega_j,\cdot)^2 u^h_{n-1}(\omega_j,\cdot)$.
\item Set $Q^h_{n+1}(\omega_j,\cdot)=-\i k \eta(\omega_j,\cdot) u^h_n(\omega_j,\cdot)$.
\item Endfor
\end{description}
\item Set $\Psi^h_N(\cdot) \leftarrow \Psi^h_N(\cdot) +\frac{1}{M} U^h_N(\omega_j,\cdot)$.
\item Endfor
\end{description}
\item Output $\Psi^h_N(\cdot)$.
\end{description}

\medskip
We remark that $\Phi^h_n$, defined in \eqref{eq4.26}, does not appear in the algorithm.
But it is easy to see that
\begin{equation}\label{eq5.1}
\Psi^h_N 
=\Phi^h_0 +\veps\Phi^h_1+\veps^2\Phi^h_2+\cdots +\veps^{N-1}\Phi^h_{N-1}.
\end{equation}

It is also easy to see that computationally the most expensive steps in the above
algorithm are those in the inside loop. In each step of the loop, one is required to solve a
large (especially for large $k$), ill-conditioned, indefinite and non-Hermitian complex
linear system. It is well-known that none of iterative methods works well for solving such 
a linear system (cf. \cite{Gander12}). Moreover, the algorithm requires one to solve a total of
$MN$ numbers of such complex linear systems. Such a task is not feasible on most of
present day computers. But, instead of using such a brute force approach,
we notice that all these $MN$ complex linear systems share the {\em same} constant
coefficient matrix. The systems only differ
in their right-hand side vectors! This is an ideal setup for using the LU
direct solver. Namely, we only need to perform one LU decomposition of the
coefficient matrix and save it. The decomposition can be re-used to solve the
remaining $MN-1$ complex linear systems by performing $MN-1$ sets of forward
and backward substitutions. This indeed is the main advantage of the
numerical procedure proposed in this paper.

The computational complexity of the above algorithm can be calculated as follows.
Let $h$ denote the mesh size of $\cT_h$ and $K:=\frac{1}{h}$ (assume it is a positive
integer). Then the (common) coefficient matrix appeared in the algorithm has the
size $O(K^d\times K^d)$, where $d$ denotes the spatial dimension of the domain $D$.
Thus, one LU decomposition requires $O(\frac{3K^{3d}}2)$ multiplications/divisions.
All $(MN-1)$ sets of forward and backward substitutions contribute
$O(MNK^d)$ multiplications/divisions. Since $N$ is a relatively small number
in practice, it can be treated as a constant. If we set $M=K^d$, which means
that the number of realizations is proportional to the number of mesh points
in $\cT_h$,  then $O(MNK^d)=O(K^{2d})$, which is still a lower order term
compared to $O(\frac{3K^{3d}}2)$. In such a practical scenario, the total
cost for implementing the above Main Algorithm is still comparable to
that of solving one deterministic Helmholtz problem by the LU direct solver.
Even if extremely large number of realizations $M=K^{2d}$ is used, the total
cost for implementing the above Main Algorithm only amounts to 
solving a few deterministic Helmholtz problem by the LU direct solver.
As a comparison, we note that if a brute force Monte Carlo method is used to solve 
\eqref{eq1.1}--\eqref{eq1.2}, it requires $O(\frac{3K^{3d}M}2)$ many multiplications/divisions.
Finally, we remark that the outer loop of the Main Algorithm
can be naturally implemented in parallel.

\subsection{Convergence analysis}\label{sec-5.2}
In this subsection, we shall combine the error estimates which we have derived in the previous
subsections for various steps in the Main Algorithm to obtain error estimates for
the global error $\E(u^\veps)-\Psi^h_N$.  To this end, we notice that
$\E(u^\veps)-\Psi^h_N$ can be decomposed as
\[
\E(u^\veps)-\Psi^h_N=\bigl(\E(u^\veps)-\E(U^\veps_N)\bigr)
+ \bigl( \E(U^\veps_N)- \E(U^h_N)\bigr) +\bigl( \E(U^h_N)-\Psi^h_N \bigr).
\]
Clearly, the first term on the right-hand side measures the finite modes representation
error, the second term measures the spatial discretization error, and the
third term represents the statistical error due to the Monte Carlo method.

First, by \eqref{eq3.16}  the finite modes representation error can be bounded as follows:
\begin{equation}\label{eq5.2}
\E(\norm{u^\veps -U^\veps_N}_{H^j(D)}^2)
\leq \frac{C_0\sigma^{2N}}{4(1+k)^2} \Bigl(k^j +\frac{1}{k} \Bigr)^4
\E(\|f\|_{L^2(D)}^2), \quad j=0,1.
\end{equation}
Where $\sigma:=4\veps C_0^{\frac12} (1+k)$.

Next, we note that
\[
U^h_N-\Psi^h_N =\sum_{n=0}^{N-1} \veps^n \bigl(u^h_n -\Phi^h_n \bigr).
\]
Then by \eqref{eq4.31} we bound the statistical error as follows:
\begin{align}\label{eq5.30}
&\E\bigl(\norm{E(U^h_N) -\Psi^h_N}_{L^2(D)}\bigr)
\leq \sum_{n=0}^{N-1} \veps^n \E\bigl(\|E(u^h_n) -\Phi^h_n\|_{L^2(D)} \bigr) \\
&\hskip 1in
\leq \frac{1}{\sqrt{M}} \Bigl(\frac{1}{k}+\frac{1}{k^2}\Bigr)
\|f\|_{L^2(\Ome,L^2(D)} \sum_{n=0}^{N-1} \veps^n \hat{C}(n,k)^\frac12 \nonumber \\
&\hskip 1in
\leq \frac{\hat{C}_0}{2\sqrt{M}} \Bigl(\frac{1}{k}+\frac{1}{k^2}\Bigr)
\|f\|_{L^2(\Ome,L^2(D)} \sum_{n=0}^{N-1} 4^n\veps^n \hat{C}_0^n(1+k)^n \nonumber \\
&\hskip 1in
\leq \frac{\hat{C}_0}{2\sqrt{M}} \Bigl(\frac{1}{k}+\frac{1}{k^2}\Bigr)
\|f\|_{L^2(\Ome,L^2(D)} \cdot \frac{1}{1-\hat{\sigma}},  \nonumber
\end{align}
where $\hat{\sigma}:=4\veps \hat{C}_0 (1+k)<1$.

Similarly, by \eqref{eq4.32} we get
\begin{align}\label{eq5.31}
\E\bigl(\norm{E(U^h_N)-\Psi^h_N}_{1,h,D}\bigr)
&\leq \sum_{n=0}^{N-1} \veps^n \E\bigl(\|E(u^h_n) -\Phi^h_n\|_{1,h,D} \bigr) \\
&\leq \frac{\hat{C}_0}{2\sqrt{M}} \Bigl(1+\frac{1}{k}\Bigr)
\|f\|_{L^2(\Ome,L^2(D)} \cdot \frac{1}{1-\hat{\sigma}}.  \nonumber
\end{align}

Finally, to bound the spatial discretization error, we recall that $u^h_n\in V^h_r$
is defined by (cf. \eqref{eq4.22}) 
\begin{equation}\label{eq5.3}
a_h\bigl( u^h_n, \psi^h\bigr) = \bigl(S^h_n, \psi^h \bigr)_D
+\langle Q^h_n, \psi^h \rangle_{\p D}  \qquad\forall v_h\in V^h_r, \,\mbox{ a.s.}
\end{equation}
for $n\geq 0$. We also define $\tilde{u}^h_n\in V^h_r$ for $n\geq 0$ by
\begin{equation}\label{eq5.4}
a_h\bigl( \tilde{u}^h_n, \psi^h\bigr) = \bigl(S_n, \psi^h \bigr)_D 
+\langle Q_n, \psi^h \rangle_{\p D} \qquad\forall \psi^h\in V^h_r,
\,\mbox{ a.s.}
\end{equation}
Notice that the difference between $u^h_n$ and $\tilde{u}^h_n$ is that $S^h_n$
and $Q^h_n$ are used in \eqref{eq5.3} while $S_n$ and $Q_n$ are used in \eqref{eq5.4}. 
Corollary \ref{cor4.1} guarantees that $\{u^h_n\}$ and $\{\tilde{u}^h_n\}$ are 
uniquely defined.

It follows from Theorem \ref{thm4.2} (ii) that for $k^3h^2r^{-2} =O(1)$ there hold
\begin{align}\label{eq5.5}
&\E\bigl(\norm{u_n-\tilde{u}^h_n}_{1,h,D}\bigr)
\leq \frac{ \tilde{C}_0(r+k^2 h) h^{\mu-1}}{r^s}\,
\E\bigl(\norm{u_n}_{H^{s}(D)}\bigr), \\
&\E\bigl(\norm{u_n-\tilde{u}^h_n}_{L^2(D)} 
+\norm{u_n-\tilde{u}^h_n}_{L^2(\p D)}  \bigr)
\leq \frac{\tilde{C}_0 kh^\mu }{r^s} \E\bigl(\norm{u_n}_{H^{s}(D)}\bigr). \label{eq5.6}
\end{align}
Where $\mu=\min\{r+1, s\}$. 

To bound $\tilde{u}^h_n-u^h_n$, we subtract \eqref{eq5.3} from \eqref{eq5.4} to get
\[
a_h\bigl( \tilde{u}^h_n-u^h_n, \psi^h\bigr) = \bigl(S_n-S^h_n, \psi^h \bigr)_D
+ \langle Q_n-Q^h_n, \psi^h \rangle_{\p D}
\qquad\forall \psi^h\in V^h_r, \,\mbox{ a.s.}
\]
Then by Theorem \ref{thm4.1} (ii) we get
\begin{align}\label{eq5.7}
&\E\Bigl( k\|\tilde{u}^h_n-u^h_n\|_{L^2(D)} + k\|\tilde{u}^h_n-u^h_n\|_{L^2(\p D)}
+\|\tilde{u}^h_n-u^h_n\|_{1,h,D}\Bigr) \\
&\hskip 0.7in
\leq \hat{C}_0 \Bigl(1+ \frac{1}{k} \Bigr)
\E \Bigl( \|S_n-S^h_n\|_{L^2(D)} + \|Q_n-Q^h_n\|_{L^2(\p D)} \Bigr) \nonumber \\
&\hskip 0.7in
\leq 2\tilde{C}_0 k(k+1) E\Bigl( 2\|u_{n-1}-u^h_{n-1}\|_{L^2(D)} 
+\|u_{n-2}-u^h_{n-2}\|_{L^2(D)} \nonumber \\
&\hskip 1.8in
+\frac{1}{k} \|u_{n-1}-u^h_{n-1}\|_{L^2(\p D)} \Bigr).  \nonumber
\end{align}
It follows from the triangle inequality, \eqref{eq5.5}-\eqref{eq5.7} and
the inverse inequality that
\begin{align}\label{eq5.8}
&\E \bigl(\|u_n -u^h_n\|_{L^2(D)} + \|u_n-u^h_n\|_{L^2(\p D)}  \bigr) \\
&\qquad
\leq \E\Bigl( \|\tilde{u}^h_n-u^h_n\|_{L^2(D)}
+ \|\tilde{u}^h_n-u^h_n\|_{L^2(\p D)} 
+ \|u_n-\tilde{u}^h_n\|_{L^2(D)}  \nonumber \\
&\hskip 1.6in
+ \|u_n-\tilde{u}^h_n\|_{L^2(\p D)} \Bigr) \nonumber \\
&\qquad
\leq 2\tilde{C}_0(k+1) E\Bigl( 2\|u_{n-1}-u^h_{n-1}\|_{L^2(D)}
+\|u_{n-2}-u^h_{n-2}\|_{L^2(D)} \nonumber\\
&\hskip 0.9in
+\frac{1}{k} \|u_{n-1}-u^h_{n-1}\|_{L^2(\p D)} \Bigr)
+\frac{\tilde{C}_0 kh^\mu }{r^s} \E\bigl(\norm{u_n}_{H^{s}(D)}\bigr),\nonumber \\
&\E\bigl(\|u_n-u^h_n\|_{1,h,D} \bigr)
\leq \E\bigl(\|\tilde{u}^h_n-u^h_n\|_{1,h,D} 
+ \|u_n-\tilde{u}^h_n\|_{1,h,D} \bigr) \label{eq5.9} \\
&\qquad
\leq Ch^{-1} \E\bigl( \|\tilde{u}^h_n-u^h_n\|_{L^2(D)} \bigr)
+ \E\bigl( \|u_n-\tilde{u}^h_n\|_{1,h,D} \bigr) \nonumber \\
&\qquad
\leq C\tilde{C}_0 h^{-1}(k+1) E\Bigl( 2\|u_{n-1}-u^h_{n-1}\|_{L^2(D)}
+\|u_{n-2}-u^h_{n-2}\|_{L^2(D)}  \nonumber\\
&\qquad\qquad
+\frac{1}{k}\|u_{n-1}-u^h_{n-1}\|_{L^2(\p D)}  \Bigr)
+\frac{ \tilde{C}_0(r+k^2 h) h^{\mu-1}}{r^s}\,
\E\bigl(\norm{u_n}_{H^{s}(D)}\bigr) \nonumber
\end{align}
for $n\geq 1$.

So we obtain two recursive relations between the spatial errors of consecutive
mode functions. Then we want to derive some estimates for the spatial error of
each mode function. To this end, we first notice that
\begin{align}\label{eq5.10}
&\E\bigl(\|u_{-1}-u^h_{-1}\|_{L^2(D)} \bigr)
= \E\bigl(\|u_{-1}-u^h_{-1}\|_{1,h,D} \bigr) =0.\\
&\E\bigl(\|u_0-u^h_0\|_{L^2(D)} + \|u_0-u^h_0\|_{L^2(\p D)} \bigr)
\leq \frac{\tilde{C}_0 kh^\mu}{r^s} \E\bigl(\norm{u_0}_{H^{s}(D)}\bigr)
\label{eq5.11} \\
&\E\bigl(\|u_0-u^h_0\|_{1,h,D} \bigr)
\leq \frac{ \tilde{C}_0(r+k^2 h) h^{\mu-1}}{r^s}\,\E\bigl(\norm{u_0}_{H^{s}(D)}\bigr).
\label{eq5.12}
\end{align}
The last two inequalities hold because $S_0=S^h_0$, $Q_0=0$ and $\tilde{u}^h_0=u^h_0$.
The above estimates for the spatial errors of the approximations of the
two starting mode functions allow us to derive the desired estimates from
\eqref{eq5.8} and \eqref{eq5.9} for all mode functions, which will be based on the
following simple lemma.

\begin{lemma}\label{lem5.1}
Let $\gamma,\beta > 0$ be two real numbers, $\{c_n\}_{n\geq 0}$ and
$\{\alpha_n\}_{n\geq 0}$ be two sequences of nonnegative numbers such that
\begin{equation}\label{eq5.14}
c_0\leq \gamma\alpha_0, \quad
c_n\leq \beta c_{n-1} +\gamma \alpha_n \quad \mbox{for } n\geq 1.
\end{equation}
Then there holds
\begin{equation}\label{eq5.15}
c_n\leq \gamma \sum_{j=0}^n \beta^{n-j} \alpha_j \qquad\mbox{for } n\geq 1.
\end{equation}

\end{lemma}

We omit the proof because it is trivial.

\begin{lemma}\label{lemma5.2}
Suppose $\sigma, \hat{\sigma} <1$ and $k^3h^2r^{-2}=O(1)$. Then there hold
\begin{align}\label{eq5.16}
&\E\bigl(\|u_n-u^h_n\|_{L^2(D)} + \|u_n-u^h_n\|_{L^2(\p D)}  \bigr)\\
&\hskip 1.2in
\leq \frac{\tilde{C}_0 kh^\mu }{r^s} \sum_{j=0}^n (2k+3)^{n-j}
\E\bigl(\norm{u_j}_{H^{s}(D)}\bigr). \nonumber \\
&\E\bigl(\|u_n-u^h_n\|_{1,h,D} \bigr)
\leq \frac{ C\tilde{C}_0^2k(1+k) h^{\mu-1}}{r^s}\,
\sum_{j=0}^n (2k+3)^{n-j} \E\bigl(\norm{u_j}_{H^{s}(D)}\bigr).
\label{eq5.17}
\end{align}

\end{lemma}

\begin{proof}
Define
\begin{align*}
&u_{-2}=u_{-1}=u^h_{-2}=u^h_{-1}=0, \\
&c_n :=\E\bigl(\|u_n-u^h_n\|_{L^2(D)} +\|u_{n-1}-u^h_{n-1}\|_{L^2(D)}\bigr) \\
&\hskip 0.7in
+\E\bigl(\|u_n-u^h_n\|_{L^2(\p D)} +\|u_{n-1}-u^h_{n-1}\|_{L^2(\p D)}\bigr),\\
&\beta:= 2k+3,\quad \gamma :=\frac{\tilde{C}_0 kh^\mu}{r^s},\quad
\alpha_n:= \E\bigl(\norm{u_n}_{H^{s}(D)}\bigr).
\end{align*}
Then by \eqref{eq5.8} we obtain \eqref{eq5.14}.  Hence \eqref{eq5.16} holds. 
\eqref{eq5.17} follows from combing \eqref{eq5.9} and \eqref{eq5.16}. The proof
is complete.
\end{proof}

Finally, by the definitions of $U^\veps_N$ and $U^h_n$,
\eqref{eq5.16} and \eqref{eq5.17}, we immediately have

\begin{theorem}\label{thm5.1}
Assume that $u_n\in L^2(\Ome,H^s(D))$ for $n\geq 0$. Then the
spatial error $U^\veps_N-U^h_N$ satisfies the following estimates:
\begin{align}\label{eq5.18}
&\E\bigl(\|U^\veps_N-U^h_N\|_{L^2(D)} \bigr)
\leq \frac{\tilde{C}_0 kh^\mu }{r^s} \sum_{n=0}^{N-1}
\sum_{j=0}^n \veps^n (2k+3)^{n-j} \E\bigl(\norm{u_j}_{H^{s}(D)}\bigr). \\
&\E\bigl(\|U^\veps_N-U^h_N\|_{1,h,D} \bigr) \label{eq5.19} \\
&\qquad\quad
\leq \frac{C \tilde{C}_0^2 k(1+k) h^{\mu-1}}{r^s}\,
\sum_{n=0}^{N-1}\sum_{j=0}^n \veps^n (2k+3)^{n-j} \E\bigl(\norm{u_j}_{H^{s}(D)}\bigr).
\nonumber
\end{align}

\end{theorem}

To simplify the above spatial error estimates, we need
to bound $\E(\norm{u_n}_{H^{s}(D)})$ in terms
of higher order norms of $f$. This is achievable using
\eqref{eq4.13} and the three-term recursive relation for $\{u_n\}$.
Below we only consider the case when $s=2$ and leave the general
case to the interested reader to explore.

When $s=2$, the required estimates have been obtained in \eqref{eq3.6a}.
Consequently, we have

\begin{theorem}\label{thm5.2}
Let $s=2$. Assume that $u_n\in L^2(\Ome,H^2(D))$ for $n\geq 0$ and $\veps= O(k^{-1})$.
Then there hold
\begin{align}\label{eq5.20}
\E\bigl(\|U^\veps_N-U^h_N\|_{L^2(D)} \bigr)
&\leq C_3(N,k,\veps)\,  h^2 \|f\|_{L^2(\Ome,L^2(D))}, \\
\label{eq5.21}
\E\bigl(\|U^\veps_N-U^h_N\|_{1,h,D} \bigr)
&\leq C_4(N,k,\veps)\, h \|f\|_{L^2(\Ome,L^2(D))},
\end{align}
where
\begin{align}\label{eq5.20a}
&C_3(N,k,\veps):=\frac{\tilde{C}_0 k}{r^2}
\cdot \frac{C_0(k^3+1)}{k^2(2\sqrt{C_0}-1)}
\cdot \frac{1-\bigl(2\sqrt{C_0} (2k+3)\veps \bigr)^N }{1-2\sqrt{C_0} (2k+3)\veps }, \\
\label{eq5.21a}
&C_4(N,k,\veps):= \frac{C\tilde{C}_0^2 k(1+k)}{r^2}
\cdot \frac{C_0(k^3+1)}{k^2(2\sqrt{C_0}-1)}
\cdot \frac{1-\bigl(2\sqrt{C_0} (2k+3)\veps \bigr)^N }{1-2\sqrt{C_0} (2k+3)\veps }.
\end{align}

\end{theorem}

\begin{proof}
By \eqref{eq3.6a} and the definition of $C(j,k)$ we get
\begin{align*}
&\sum_{n=0}^{N-1}\sum_{j=0}^n \veps^n (2k+3)^{n-j} \E\bigl(\norm{u_j}_{H^{s}(D)}\bigr)\\
&\quad
\leq \Bigl(k+\frac{1}{k^2}\Bigr) \|f\|_{L^2(\Ome,L^2(D))}
\sum_{n=0}^{N-1}\sum_{j=0}^n \veps^n (2k+3)^{n-j} C(j,k)^{\frac12} \\
&\quad
=\frac{C_0^{\frac12}(k^3+1)}{2k^2} \|f\|_{L^2(\Ome,L^2(D))}
\sum_{n=0}^{N-1}\sum_{j=0}^n \veps^n 4^j C_0^{\frac{j}{2}} (1+k)^j (2k+3)^{n-j}\\
&\quad
\leq \frac{C_0(k^3+1)}{k^2(2\sqrt{C_0}-1)}
\cdot \frac{1-\bigl(2\sqrt{C_0} (2k+3)\veps \bigr)^N }{1-2\sqrt{C_0} (2k+3)\veps }
\|f\|_{L^2(\Ome,L^2(D))}.
\end{align*}
The above inequality and \eqref{eq5.18} yield \eqref{eq5.20}. Similarly, the above
inequality and \eqref{eq5.19} give \eqref{eq5.21}.  The proof is complete.
\end{proof}

Combining \eqref{eq5.2}--\eqref{eq5.31}, \eqref{eq5.20},  \eqref{eq5.21},
\eqref{eq4.31} and \eqref{eq4.32} we get

\begin{theorem}\label{thm5.5}
Under the assumptions that $u_n\in L^2(\Ome,H^2(D))$ for $n\geq 0$, $k^3h^2r^{-2}=O(1)$
and $\veps= O(k^{-1})$, there hold
\begin{align}\label{eq5.23}
\E\bigl( \|\E(u^\veps) - \Psi_N^h\|_{L^2(D)}) \leq C_1 \veps^N +C_2 h^2 + C_3 M^{-\frac12},\\
\E\bigl( \|\E(u^\veps) - \Psi_N^h\|_{H^1(D)}) \leq C_4 \veps^N +C_5 h + C_6 M^{-\frac12},\label{eq5.24}
\end{align}
where $C_j=C_j(C_0,\hat{C}_0,k,\veps)$ are positive constants for $j=1,2,\cdots, 6$.
\end{theorem}

\section{Numerical experiments}\label{sec-6}
In this section we present a series of numerical experiments in order to accomplish the following:
\begin{itemize}
\item compare our MCIP-DG method using the multi-modes expansion to a classical MCIP-DG method,
\item illustrate examples using our MCIP-DG method in which the perturbation parameter 
$\veps$ satisfies the constraint required by the convergence theory,
\item illustrate examples using our MCIP-DG method in which the perturbation 
parameter constraint is violated, 
\item illustrate examples using our MCIP-DG method in which we allow $\eta$ to be large in magnitude.
\end{itemize}

In all our numerical experiments we use the spatial domain $D =(-0.5,0.5)^2$. To partition
$D$ we use a uniform triangulation $\mathcal{T}_h$.  For a positive integer $n$,   
$\mathcal{T}_{1/n}$ denotes the triangulation of $D$ consisting of $2n^2$ congruent isosceles 
triangles with side lengths $1/n,1/n,$ and $\sqrt{2}/n$.  Figure \ref{fig:MeshExample} gives the sample 
triangulation $\mathcal{T}_{1/10}$.

\begin{figure}[htb]
\centering
\includegraphics[scale=0.3]{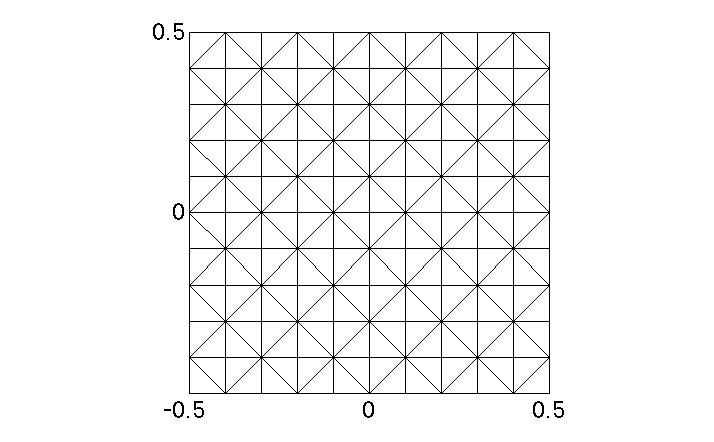}
\caption{Triangulation $\mathcal{T}_{1/10}$ \label{fig:MeshExample}}
\end{figure}

To implement the random noise $\eta$, we note that $\eta$ only appears in the integration 
component of our computations.  Therefore, we made the choice to implement $\eta$ only 
at quadrature points of the triangulation.  To simulate the random media, we let $\eta$ 
be an independent random number chosen from a uniform distribution on some closed interval 
at each quadrature point. Figure \ref{fig:RandomMedia} shows an example of such random media.

\begin{figure}[htb]
\centerline{\includegraphics[scale = .25]{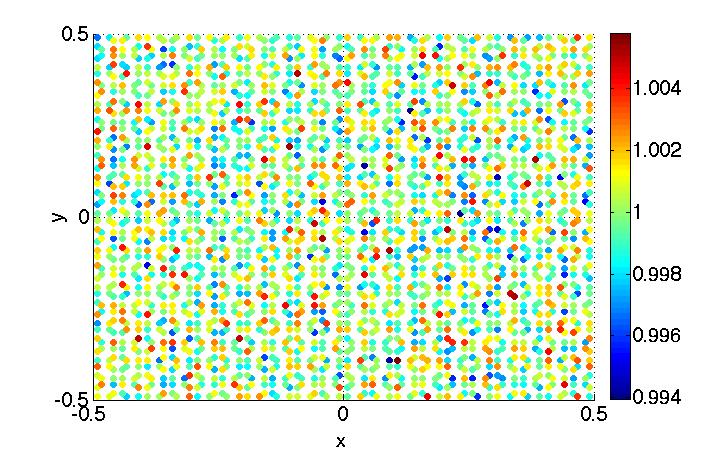} 
\includegraphics[scale = .25]{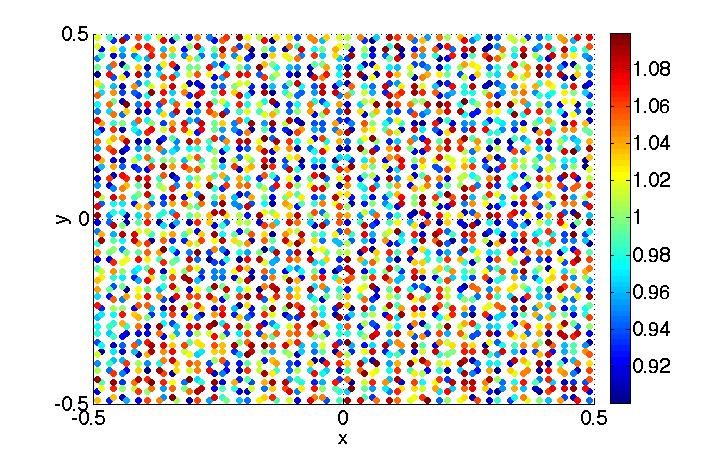}}
\caption{Discrete average media $\frac{1}{M}\sum_{j=1}^{M} \alpha(\ome_j,\cdot)$ (left) 
and a sample media $\alpha(\ome,\cdot)$ (right) computed for $h = 1/20$, $\veps = 0.1$, 
$\eta(\cdot,x) \thicksim \mathcal{U}[-1,1]$, and $M = 1000$} \label{fig:RandomMedia}
\end{figure}

\subsection{MCIP-DG with multi-modes expansion compared to classical MCIP-DG}
The goal of this subsection is to verify the accuracy and efficiency of the proposed 
MCIP-DG with the multi-modes expansion.  As a benchmark we compare this method to 
the classical MCIP-DG (i.e. without utilizing the multi-modes expansion). Throughout this section 
$\tilde{\Psi}^h$ is used to denote the computed approximation to $\mathbb{E}(u)$ using the 
classical MCIP-DG. 

In this subsection we set $f = 1$, $k = 5$, $1/h = 50$, $M = 1000$, and $\veps = 1/(k+1)$. 
Here $\veps$ is chosen with the intent of satisfying the constraint set by the convergence 
theory in the preceding section. $\eta$ is sampled as described above from a uniform 
distribution on the interval $[0,1]$.  $\Psi^h_N$ is computed for $N = 1,2,3,4,5$. 

In our first test we compute $\| \Psi^h_N - \tilde{\Psi}^h\|_{L^2(D)}$. The results are 
displayed in Figure \ref{fig:ModesVsFullMonte}.  As expected, we find that the difference 
between $\Psi^h_N$ and $\tilde{\Psi}^h$ is very small.  We also observe that we are 
benefited more by the first couple modes while the help from the later modes is relatively small. 

To test the efficiency of our MCIP-DG method with multi-modes expansion, we compare 
the CPU time for computing $\Psi^h_N$ and $\tilde{\Psi}^h$.  Both methods are implemented 
on the same computer using Matlab. Matlab's built-in LU factorization is called 
to solve the linear systems.  The results of this test are shown in 
Table \ref{table:ModesVsFullMonte}.  As expected, we find that the use of the multi-modes 
expansion improves the CPU time for the computation considerably.  In fact, the table shows 
that this improvement is an order of magnitude.  Also, as expected, as the number of modes 
used is increased the CPU time increases in a linear fashion.

\begin{figure}[htb]
\centering
\includegraphics[scale=0.25]{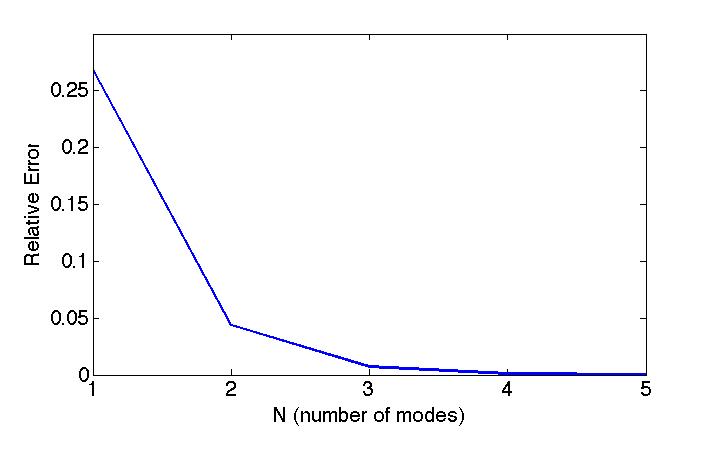}
\caption{$L^2$-norm error between $\Psi^h_N$ computed using MCIP-DG 
with the multi-modes expansion and $\tilde{\Psi}^h$ computed using the classical MCIP-DG.}
\label{fig:ModesVsFullMonte}
\end{figure}
\begin{table}[htb]
\centering
\begin{tabular}{| c | c |}
	\hline
	Approximation & CPU Time (s) \\
	\hline
	$\tilde{\Psi}^h$ & $3.4954 \times 10^5$ \\
	\hline
	$\Psi^h_1$ & $1.0198 \times 10^4$ \\
	\hline
	$\Psi^h_2$ & $2.0307 \times 10^4$ \\
	\hline
	$\Psi^h_3$ & $3.0037 \times 10^4$ \\
	\hline
	$\Psi^h_4$ & $3.9589 \times 10^4$ \\
	\hline
	$\Psi^h_5$ & $4.9011 \times 10^4$ \\
	\hline
\end{tabular}
\caption{CPU times required to compute the MCIP-DG multi-modes approximation $\Psi^h_N$ 
and classical MCIP-DG approximation $\tilde{\Psi}^h$.} \label{table:ModesVsFullMonte}
\end{table}

\subsection{More numerical tests}
The goal of this subsection is to demonstrate the approximations obtained by our 
MCIP-DG method with multi-modes expansion using different magnitudes of parameter $\veps$ 
and different magnitudes of the random noise $\eta$.  We only consider the case $0 < \veps < 1$ 
in order to legitimize the series expansion $u^{\veps}$.  With this in mind, we then
increase the magnitude of $\eta$ to simulate examples with large noise.  Similar to the 
numerical experiments from \cite{Feng_Wu09}, we choose the function 
$f = \sin\big(k \alpha(\ome, \cdot) r \big)/r$, where $r$ is the radial distance from 
the origin and $\alpha(\ome, \cdot)$ is implemented as described in the beginning 
of this section.  Since our intention is to observe what happens as we vary $\veps$
and $\eta$, we fix $k = 50$, $h = 1/100$, and $M = 1000$. 

In Figures \ref{fig:EpsTest1a} and \ref{fig:EpsTest1b}, we set $\veps = 0.02$ and 
$|\eta|\leq 1$ with the intent of observing the constraints set in the convergence 
theory from the preceding section. In Figure \ref{fig:EpsTest1a} we present plots of the 
magnitude of the computed mean $\Re \big(\Psi^h_3\big)$ and a computed sample 
$\Re \big(U^h_3 \big)$, respectively, over the whole domain $D$.  Figure \ref{fig:EpsTest1b} 
gives the plots of a cross section of the computed mean $\Re \big(\Psi^h_3\big)$ and 
a computed sample $\Re \big(U^h_3 \big)$, respectively, over the line $y = x$.  
In this first example we observe that the computed sample does not differ greatly 
from the computed mean because $\veps$ is very small.

In Figures \ref{fig:EpsTest3a}--\ref{fig:EpsTest10b}, we fix $|\eta| \leq 1$ and 
increase $\veps$ past the constraint established in the preceding convergence theory.  
As expected, we see that as $\veps$ increases the computed sample differs more from the 
computed mean. We also observe that as $\veps$ increases the phase of the wave remains 
relatively intact but the magnitude of the wave becomes more uniform. 

In Table \ref{table:ModesVSFullMonteError1} the relative error (measured in 
the $L^2$-norm) between the multi-modes approximation $\Psi^h_N$ and the classical
Monte Carlo approximation $\Psi^h$ is given for $\veps = 0.02,0.1,0.5,0.8$.  In this 
table only three modes (i.e., $N = 3$) are used.  Recall that the convergence theory in this case 
only holds for $\veps$ on the order of the first value $0.02$.  That being said, we observe 
that the approximations corresponding to $\veps = 0.1$ and $\veps = 0.5$ are relatively 
close to those obtained using the classical Monte Carlo method.  Another observation that can 
be made from Table \ref{table:ModesVSFullMonteError1} is that as $\veps$ increases the 
relative error increases.  This is expected from the convergence theory.

Recall that the error predicted in the convergence theory can be bounded by a term with the 
factor $\veps^N$.  Thus for $\veps$ relatively large, one must use more modes to decrease the error.  
Keeping this in mind, Table \ref{table:ModesVSFullMonteError2} records the relative error 
(measured in the $L^2$-norm) between the multi-modes approximation $\Psi^h_N$ and the classical 
Monte Carlo approximation $\Psi^h$ is given for $\veps = 0.5,0.8$ and $N = 4, 5, 6, 7$.  
At this point, we observe that the relative error decreases as N increases when $\veps = 0.5$.  
On the other hand, the relative error increases as $N$ increases when $\veps = 0.8$. 
From Tables \ref{table:ModesVSFullMonteError1} and \ref{table:ModesVSFullMonteError2} 
we observe that multi-modes expansion $\Psi^h_N$ is relatively accurate (measured against 
an approximation from the classical Monte Carlo method) even in cases when $\veps$ 
does not satisfy the constraint set forth in the convergence theory.  We also observe 
that when $\veps$ becomes too large, the multi-modes expansion no longer agrees with 
the classical Monte Carlo method.

\begin{table}[htb]
\centering
\begin{tabular}{| c || c | c | c | c |}
        \hline
        $\veps$ & $0.02$ & $0.1$ & $0.5$ & $0.8$ \\
        \hline
        Relative $L^2$ Error & $3.0125 \times 10^{-4}$ & $6.0073\times 10^{-4}$ &
        $0.2865$ & $1.6979$ \\
        \hline
\end{tabular}
%
\caption{$L^2$-norm relative error between the multimodes expansion approximation $\Psi^h_3$ 
and the classical Monte Carlo approximation $\tilde{\Psi}^h$. 
}\label{table:ModesVSFullMonteError1}
\end{table}

\begin{table}[htb]
\centering
\begin{tabular}{| c || c | c | c | c |}
         \hline
	$\veps$ & $N = 4$ & $N = 5$ & $N = 6$ & $N = 7$ \\
	\hline
	\hline
	$0.5$ & $0.2866$ & $0.1125$ & $0.1137$ & $0.0554$\\
	\hline
	$0.8$ & $1.7036$ & $1.6713$ & $1.6839$ & $1.7887$  \\
	\hline
\end{tabular}
\caption{$L^2$-norm relative error between the multimodes expansion approximation $\Psi^h_N$ 
and the classical Monte Carlo approximation $\tilde{\Psi}^h$.
}\label{table:ModesVSFullMonteError2}
\end{table}


\begin{figure}[htb]
\centerline{\includegraphics[scale=.25]{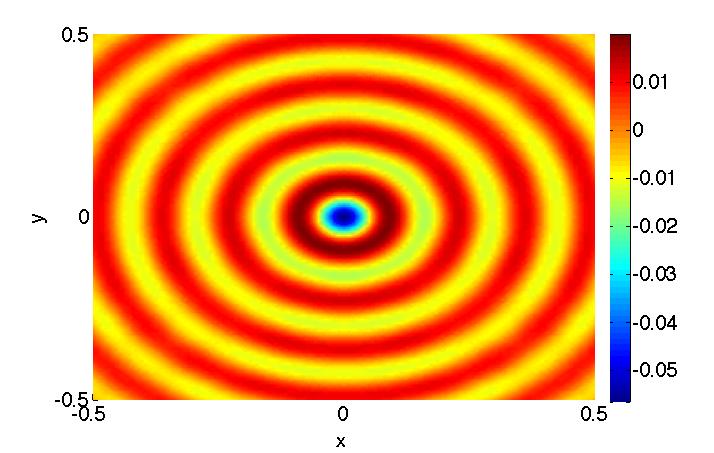} \includegraphics[scale = .25]{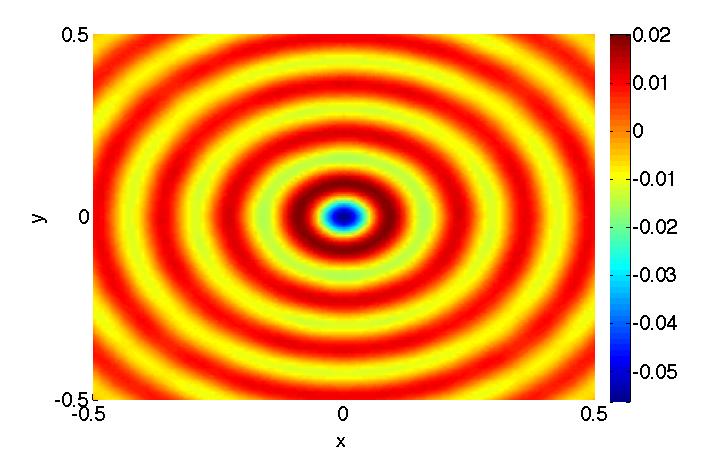}}
\caption{$\Re \big(\Psi^h_3 \big)$ (left) and $\Re \big(U^h_3\big)$(right) computed for $k = 50$, $h = 1/100$, 
$\veps = 0.02$, $\eta(\cdot,x) \thicksim \mathcal{U}[-1,1]$, and $M=1000$.}  
\label{fig:EpsTest1a}
\end{figure}

\begin{figure}[htb]
\centerline{
\includegraphics[scale = .25]{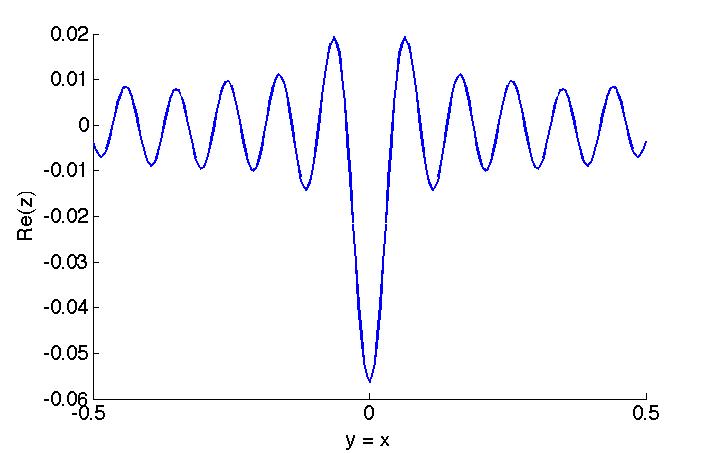} 
\includegraphics[scale = .25]{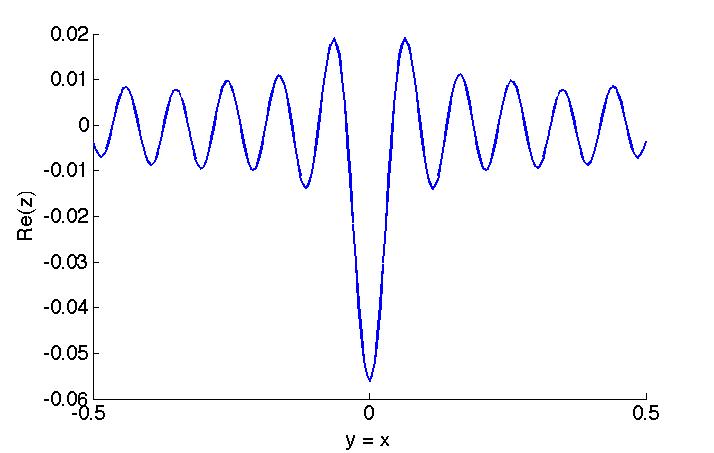}}
\caption{Cross sections of $\Re \big(\Psi^h_3 \big)$ (left) and $\Re \big(U^h_3\big)$ (right) computed for 
$k = 50$, $h = 1/100$, $\veps = 0.02$, $\eta(\cdot,x) \thicksim \mathcal{U}[-1,1]$, and
$M=1000$, over the line $y = x$.} 
\label{fig:EpsTest1b}
\end{figure}

\begin{figure}[htb]
\centerline{
\includegraphics[scale = .25]{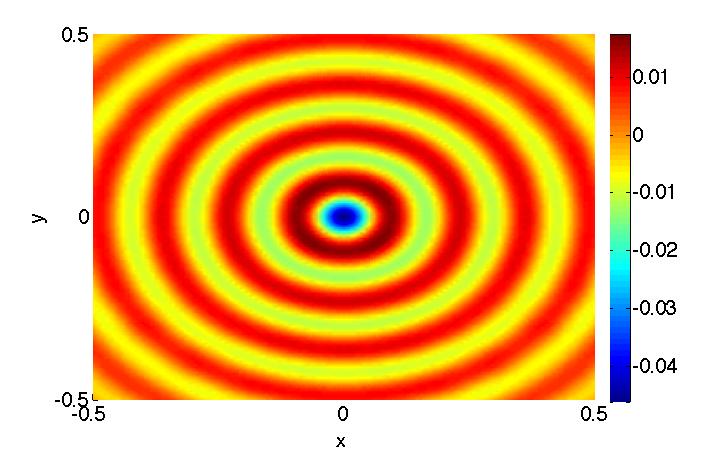} 
\includegraphics[scale = .25]{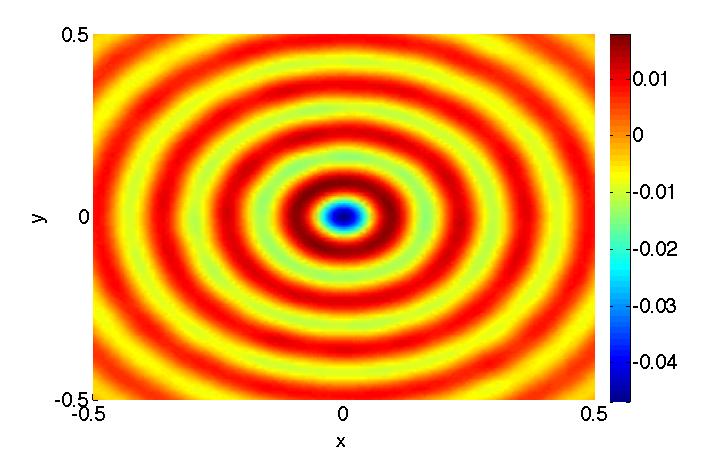} }
\caption{$\Re \big(\Psi^h_3 \big)$ (left) and $\Re \big(U^h_3\big)$ (right) 
computed for $k = 50$, $h = 1/100$, $\veps = 0.1$, $\eta(\cdot,x) \thicksim 
\mathcal{U}[-1,1]$, and $M = 1000$.} \label{fig:EpsTest3a}
\end{figure}

\begin{figure}[htb]
\centerline{
\includegraphics[scale=.25]{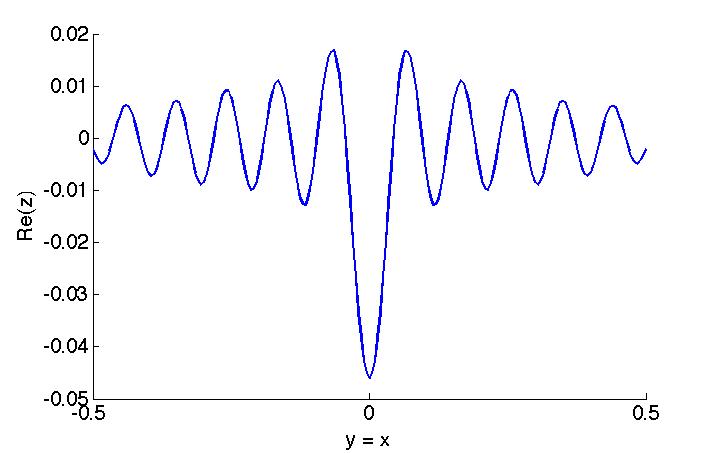}  
\includegraphics[scale = .25]{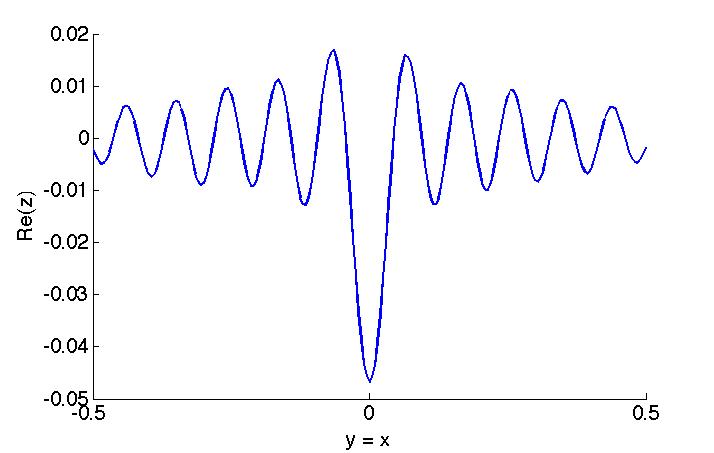}}
\caption{Cross sections of $\Re \big(\Psi^h_3 \big)$ (left) and $\Re \big(U^h_3 \big)$ (right) 
computed for $k = 50$, $h = 1/100$, $\veps = 0.1$, $\eta(\cdot,x) \thicksim 
\mathcal{U}[-1,1]$, and $M = 1000$.}  \label{fig:EpsTest3b}
\end{figure}

\begin{figure}[htb]
\centerline{
\includegraphics[scale = .25]{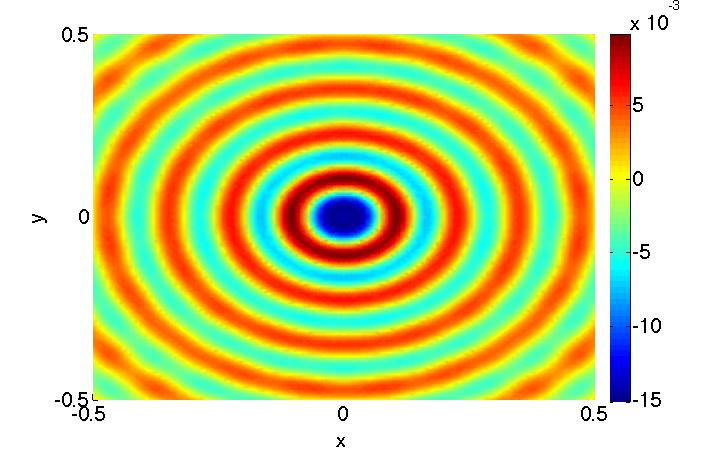} 
\includegraphics[scale = .25]{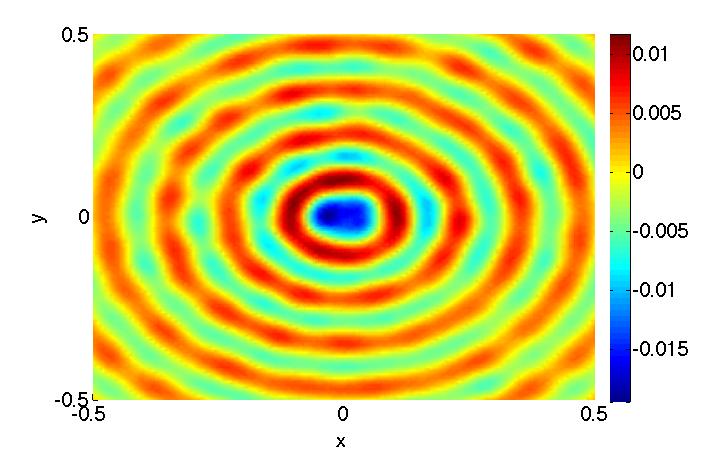}}
\caption{$\Re \big(\Psi^h_3 \big)$ (left) and $\Re \big(U^h_3\big)$ (right) computed 
for $k = 50$, $h = 1/100$, $\veps = 0.5$, $\eta(\cdot,x) \thicksim \mathcal{U}[-1,1]$, 
and $M = 1000$.}  \label{fig:EpsTest7a}
\end{figure}

\begin{figure}[htb]
\centerline{
\includegraphics[scale = .25]{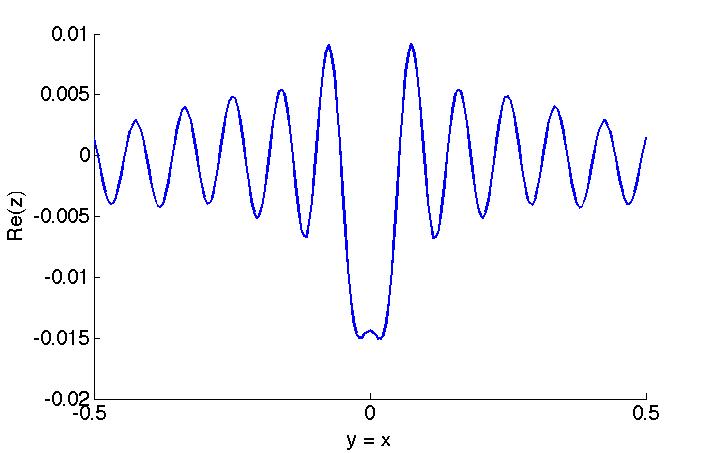} 
\includegraphics[scale = .25]{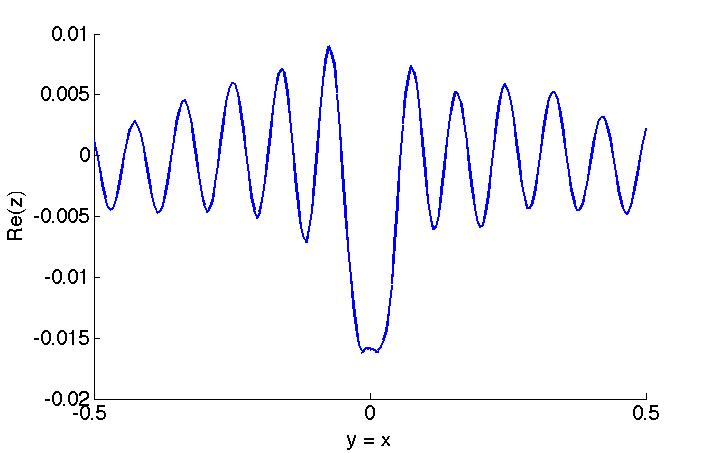}}
\caption{Cross sections of $\Re \big(\Psi^h_3 \big)$ (left) and $\Re \big(U^h_3 \big)$ (right) 
computed for $k = 50$, $h = 1/100$, $\veps = 0.5$, $\eta(\cdot,x) \thicksim \mathcal{U}[-1,1]$, 
and $M = 1000$.} \label{fig:EpsTest7b}
\end{figure}

\begin{figure}[htb]
\centerline{
\includegraphics[scale = .25]{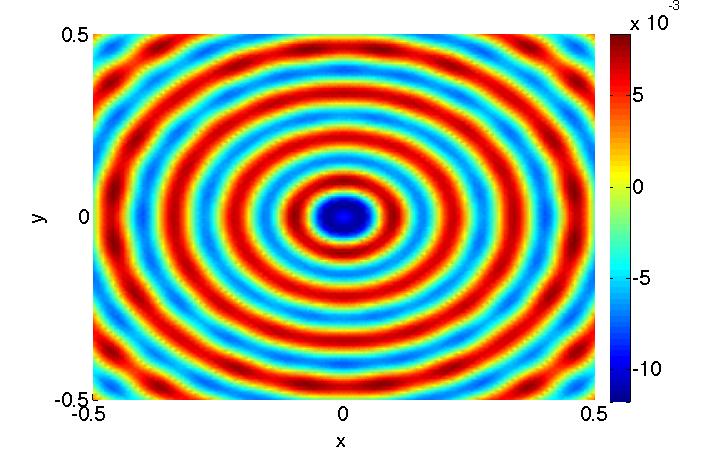} 
\includegraphics[scale = .25]{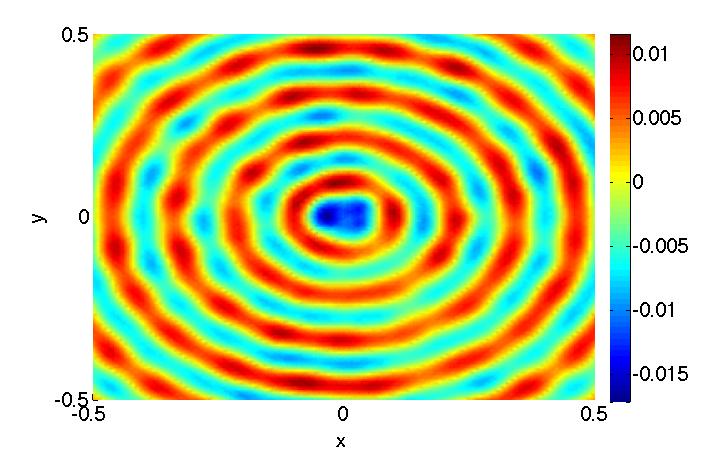}}
\caption{$\Re \big(\Psi^h_3 \big)$ (left) and $\Re \big(U^h_3 \big)$ (right) computed 
for $k = 50$, $h = 1/100$, $\veps = 0.8$, $\eta(\cdot,x) \thicksim \mathcal{U}[-1,1]$, 
and $M = 1000$.} \label{fig:EpsTest10a}
\end{figure}

\begin{figure}[htb]
\centerline{
\includegraphics[scale = .25]{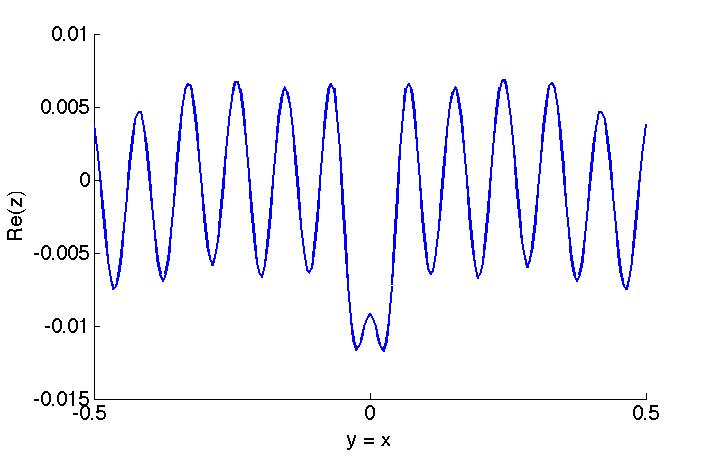} 
\includegraphics[scale = .25]{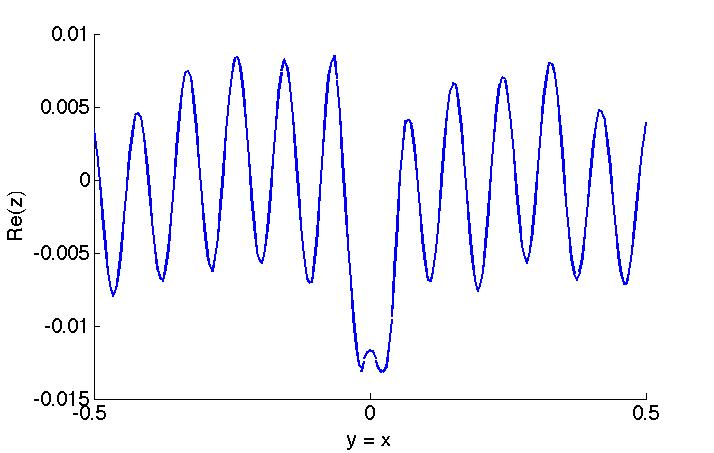}}
\caption{Cross sections of $\Re \big(\Psi^h_3 \big)$ (left) and $\Re \big(U^h_3 \big)$ (right) 
computed for $k = 50$, $h = 1/100$, $\veps = 0.8$, $\eta(\cdot,x) \thicksim \mathcal{U}[-1,1]$, 
and $M = 1000$.} \label{fig:EpsTest10b}
\end{figure}

In Figures \ref{fig:EtaTest1a}--\ref{fig:EtaTest10b}, we fix $\veps = 0.9$ and increase 
the magnitude of $\eta$. We observe that as the magnitude of random noise increases 
the difference between computed sample and computed mean increases.  We also observe that 
the phase of the computed wave remains intact until the random noise becomes too large 
(see Figures \ref{fig:EtaTest10a} and \ref{fig:EtaTest10b}). At this point we believe that 
increasing the number of samples is necessary in order to capture the mean of the large noise.

\begin{figure}[htb]
\centerline{
\includegraphics[scale = .25]{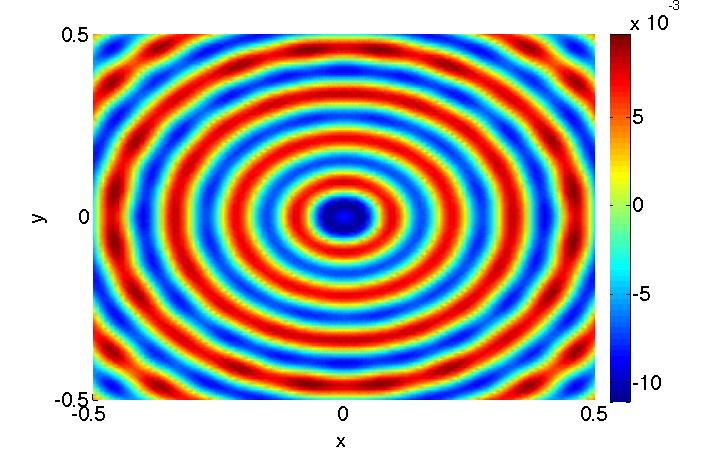} 
\includegraphics[scale = .25]{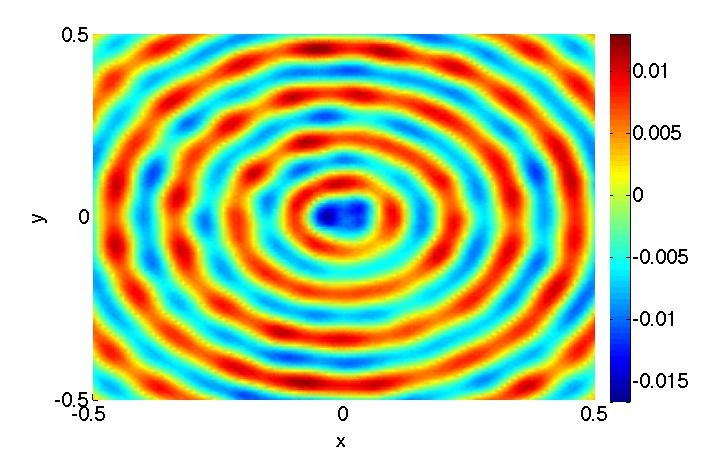}}
\caption{$\Re \big(\Psi^h_3 \big)$ (left) and $\Re \big(U^h_3 \big)$ (right) computed 
for $k = 50$, $h = 1/100$, $\veps = 0.9$, $\eta(\cdot,x) \thicksim \mathcal{U}[-1,1]$, 
and $M = 1000$.} \label{fig:EtaTest1a}
\end{figure}

\begin{figure}[htb]
\centerline{
\includegraphics[scale = .25]{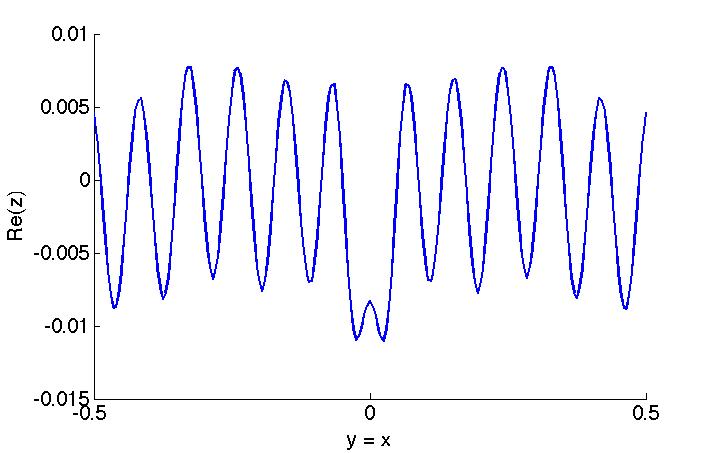} 
\includegraphics[scale = .25]{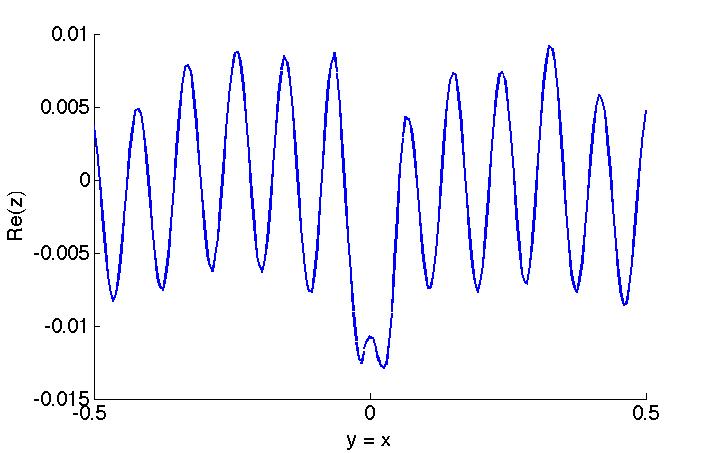}}
\caption{Cross sections of $\Re \big(\Psi^h_3 \big)$ (left) and $\Re \big(U^h_3 \big)$ (right) 
computed for $k = 50$, $h = 1/100$, $\veps = 0.9$, $\eta(\cdot,x) \thicksim \mathcal{U}[-1,1]$, 
and $M = 1000$.} \label{fig:EtaTest1b}
\end{figure}

\begin{figure}[htb]
\centerline{
\includegraphics[scale = .25]{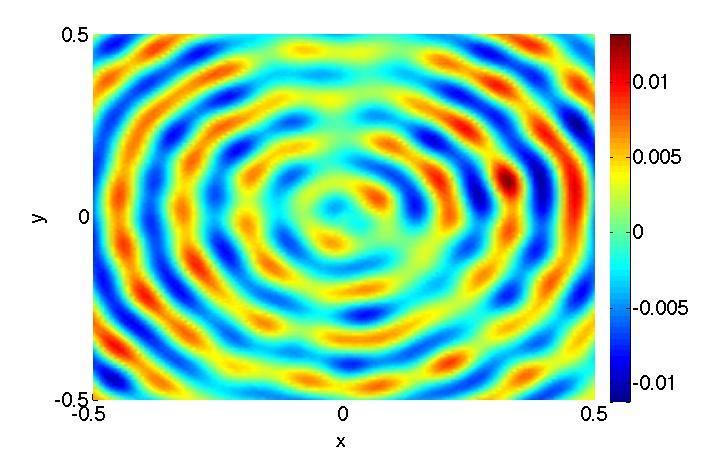} 
\includegraphics[scale = .25]{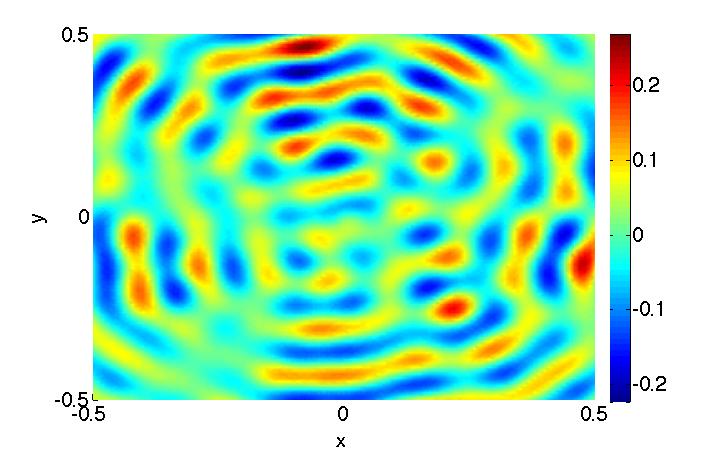}}
\caption{$\Re \big(\Psi^h_3 \big)$ (left) and $\Re \big(\Psi^h_3 \big)$ (right) computed 
for $k = 50$, $h = 1/100$, $\veps = 0.9$, $\eta(\cdot,x) \thicksim \mathcal{U}[-10,10]$, 
and $M = 1000$.} \label{fig:EtaTest3a}
\end{figure}

\begin{figure}[htb]
\centerline{
\includegraphics[scale = .25]{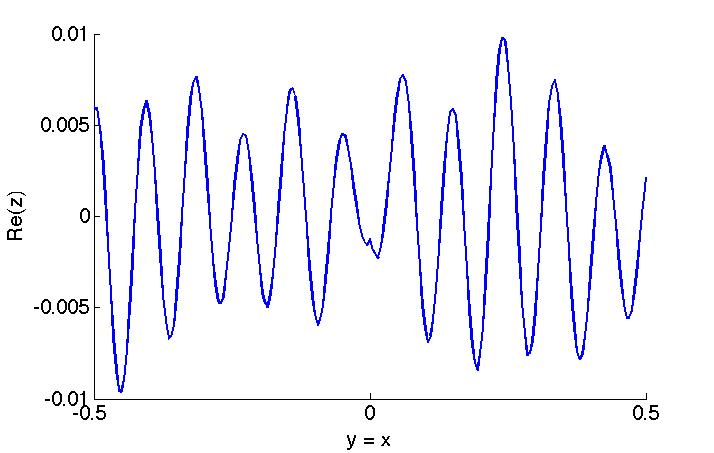} 
\includegraphics[scale = .25]{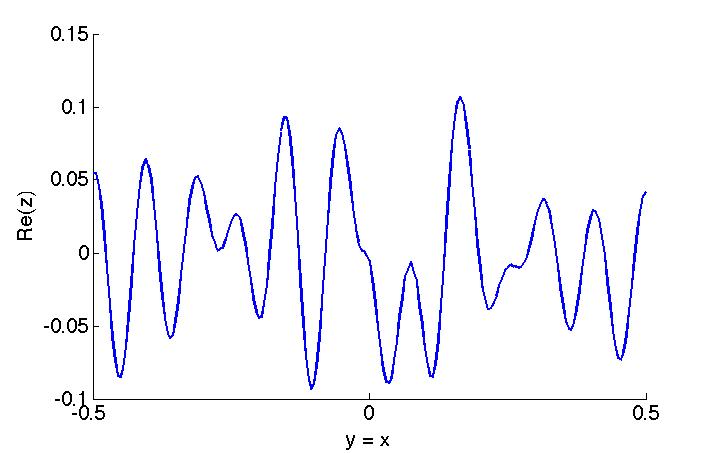}}
\caption{Cross sections of $\Re \big(\Psi^h_3 \big)$ (left) and $\Re \big(U^h_3 \big)$ (right) 
computed for $k = 50$, $h = 1/100$, $\veps = 0.9$, $\eta(\cdot,x) \thicksim \mathcal{U}[-10,10]$, 
and $M = 1000$.} \label{fig:EtaTest3b}
\end{figure}

\begin{figure}[htb]
\centerline{
\includegraphics[scale = .25]{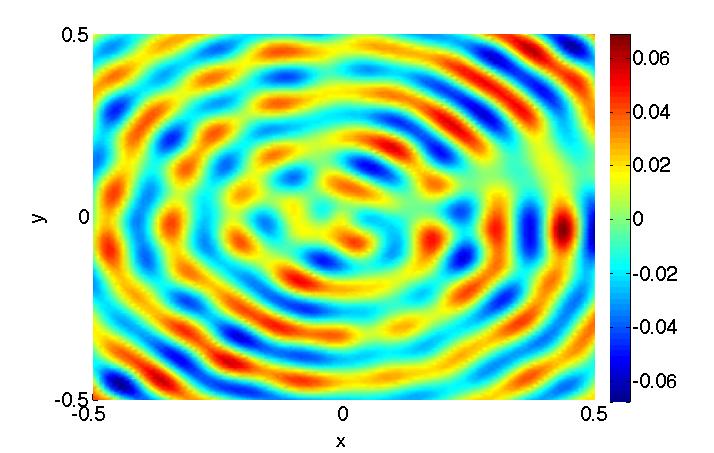} 
\includegraphics[scale = .25]{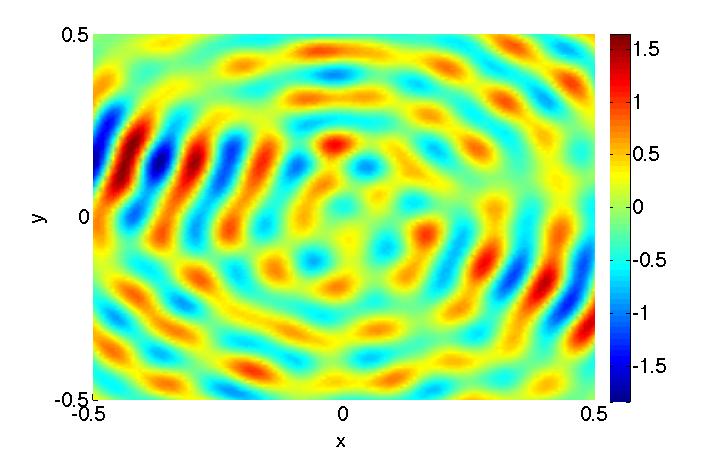}}
\caption{$\Re \big(\Psi^h_3 \big)$ (left) and $\Re \big(U^h_3 \big)$ (right) computed 
for $k = 50$, $h = 1/100$, $\veps = 0.9$, $\eta(\cdot,x) \thicksim \mathcal{U}[-25,25]$, 
and $M = 1000$.} \label{fig:EtaTest6a}
\end{figure}

\begin{figure}[htb]
\centerline{
\includegraphics[scale = .25]{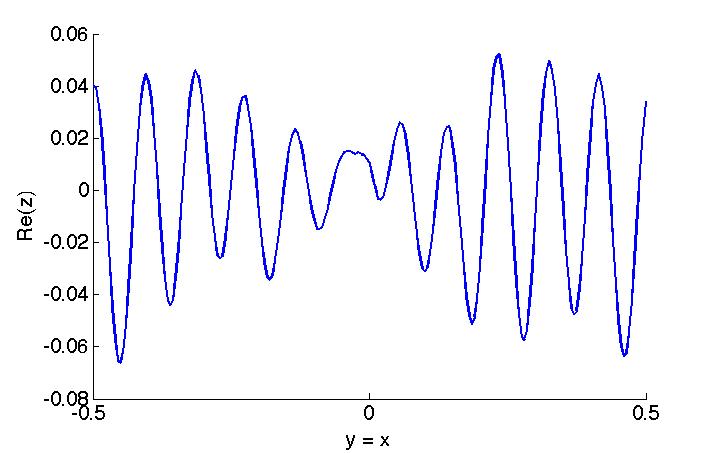} 
\includegraphics[scale = .25]{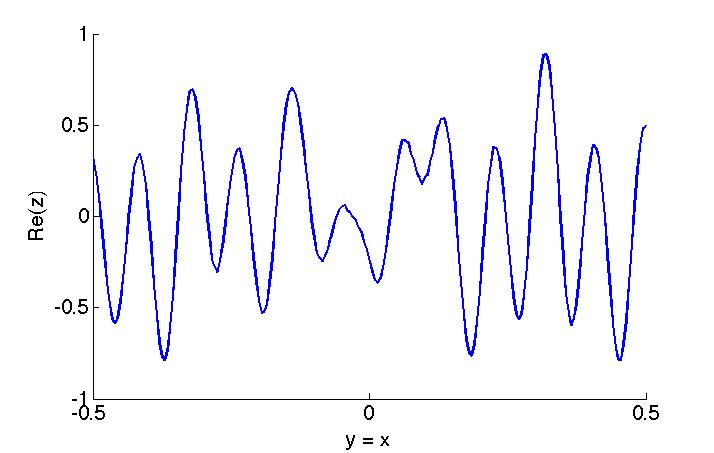}}
\caption{Cross sections of $\Re \big(\Psi^h_3 \big)$ (left) and $\Re \big(U^h_3 \big)$ (right) 
computed for $k = 50$, $h = 1/100$, $\veps = 0.9$, $\eta(\cdot,x) \thicksim \mathcal{U}[-25,25]$, 
and $M = 1000$.} \label{fig:EtaTest6b}
\end{figure}

\begin{figure}[htb]
\centerline{
\includegraphics[scale = .25]{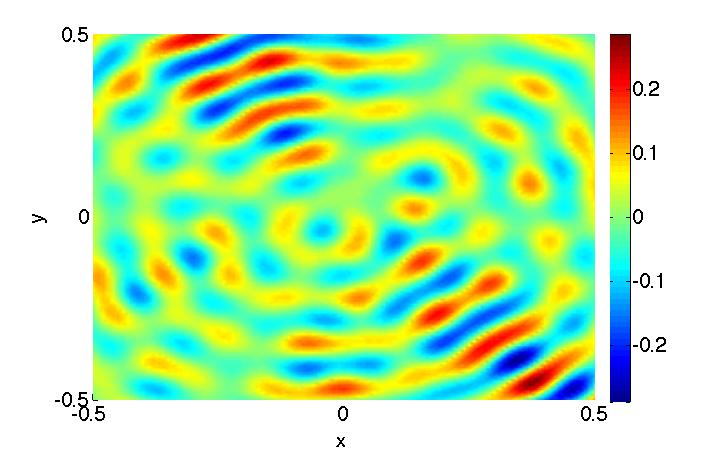} 
\includegraphics[scale = .25]{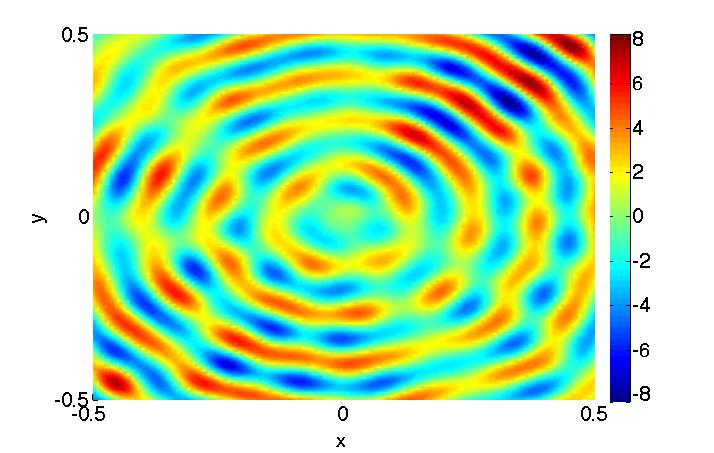}}
\caption{$\Re \big(\Psi^h_3 \big)$ (left) and $\Re \big(U^h_3 \big)$ (right) computed 
for $k = 50$, $h = 1/100$, $\veps = 0.9$, $\eta(\cdot,x) \thicksim \mathcal{U}[-50,50]$, 
and $M = 1000$.} \label{fig:EtaTest10a}
\end{figure}
\clearpage
\begin{figure}[htb]
\centerline{
\includegraphics[scale = .25]{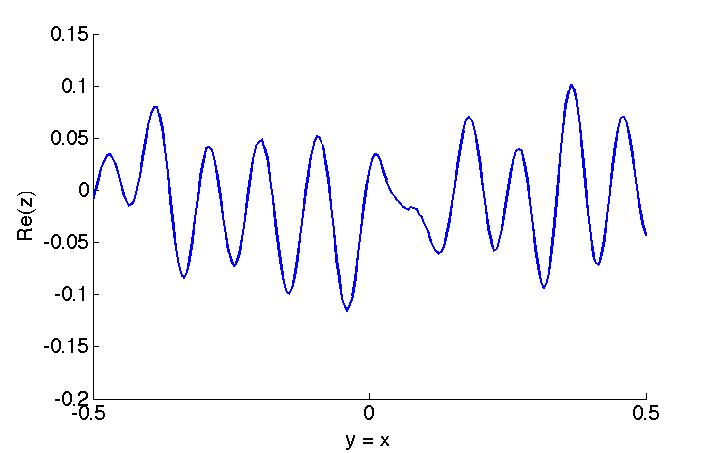} 
\includegraphics[scale = .25]{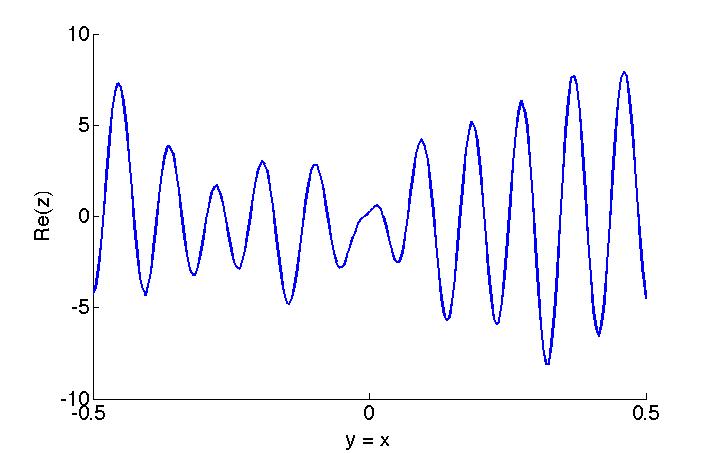}}
\caption{Cross sections of $\Re \big(\Psi^h_3 \big)$ (left) and $\Re \big(U^h_3 \big)$ (right) 
computed for $k = 50$, $h = 1/100$, $\veps = 0.9$, $\eta(\cdot,x) \thicksim \mathcal{U}[-50,50]$, 
and $M = 1000$.} \label{fig:EtaTest10b}
\end{figure}

\bigskip
{\bf Acknowledgments.} This project was initiated while both the first and second authors
were long-term visitors (as a new direction professor and a postdoc, respectively) of IMA 
at University of Minnesota in the spring of 2013. Both authors are grateful for the financial 
support and the visiting opportunity provided by IMA.



\begin{thebibliography}{99}

\bibitem{Babuska_Nobile_Tempone10}
I. Babu\v{s}ka, F. Nobile and R. Tempone.
\newblock A stochastic collocation method for elliptic partial differential
equations with random input data.
\newblock {\em SIAM Rev.},  52:317--355, 2010.

\bibitem{Babuska_Tempone_Zouraris04}
I. Babu\v{s}ka, R. Tempone and G. E. Zouraris.
\newblock Galerkin finite element approximations of stochastic elliptic partial
 differential equations.
\newblock {\em SIAM J. Numer. Anal.},  42:800--825, 2004.
		

\bibitem{Caflisch} R. Caflisch,
\newblock  Monte Carlo and quasi-Monte Carlo methods.
\newblock {\em Acta Numerica},
7:1--49, 1998.


\bibitem{Cummings_Feng06} P. Cummings and  X. Feng.
\newblock Sharp regularity coefficient estimates for complex-valued acoustic and
elastic Helmholtz equations.
\newblock {\em M$^3$AS}, 16:139--160, 2006.

\bibitem{EEU} M. Eiermann, O. Ernst, and E. Ullmann,
\newblock Computational aspects of the stochastic finite
element method.
\newblock {\em Proceedings of ALGORITMY}, 1-10, 2005.


\bibitem{em79}
B.~Engquist and A.~Majda.
\newblock Radiation boundary conditions for acoustic and elastic wave
  calculations.
\newblock {\em Comm. Pure Appl. Math.}, 32(3):314--358, 1979.

\bibitem{Feng_Wu09}
X.~Feng and H. Wu.
\newblock Discontinuous Galerkin methods for the Helmholtz equation
with large wave numbers.
\newblock {\em SIAM J. Numer. Anal.}, 47:2872--2896, 2009.

\bibitem{Feng_Wu11}
X.~Feng and H. Wu.
\newblock $hp$-Discontinuous Galerkin methods for the Helmholtz equation
with large wave numbers.
\newblock {\em Math. Comp.}, 80:1997--2024, 2011.

\bibitem{FGPS}
J. Fouque, J. Garnier, G. Papanicolaou and K. Solna,
\newblock Wave Propagation and Time Reversal in Randomly Layered Media.
\newblock Stochastic Modeling and Applied Probability, Vol. 56, Springer, 2007.

\bibitem{Gander12}
O. Ernst and M. Gander.
\newblock Why it is difficult to solve Helmholtz problems with classical iterative methods?
\newblock in {\em Numerical Analysis of Multiscale Problems,} I. Graham, T. Hou, 
O. Lakkis and R. Scheichl, Editors, pp. 325-363, Springer Verlag, 2012. 

\bibitem{Gilbarg_Trudinger01}
D.~Gilbarg, N.~S. Trudinger.
\newblock {\em Elliptic Partial Differential Equations of Second Order},
Classics in Mathematics.
\newblock Springer-Verlag, Berlin, 2001, reprint of the 1998 edition.

\bibitem{Ishimaru}
A. Ishimaru,
\newblock Wave Propagation and Scattering in Random Media. IEEE Press, New York, 1997.

\bibitem{Leis86} 
R. Leis,
\newblock Initial-Boundary Value Problems in Mathematical Physics. T\"ubner, 1986.

\bibitem{Liu_Riviere13}
K. Liu and B. Rivi\`ere.
\newblock Discontinuous Galerkin methods for elliptic partial differential
equations with random coefficients.
\newblock {\em Int. J. Computer Math.}, DOI: 10.1080/00207160.2013.784280.

\bibitem{Ronan_Sarkis}
L. Roman and M. Sarkis,
\newblock Stochastic Galerkin method for elliptic SPDEs: A white noise
approach,
\newblock {\em Discret. Contin. Dyn. S.}, 6:941-955, 2006.

\bibitem{Xiu_Karniadakis}
D. Xiu and G. Karniadakis,
\newblock The Wiener-Askey polynomial chaos for stochastic differential equations.
\newblock {\em SIAM J. Sci. Comput.}, 24:619-644, 2002.

\end{thebibliography}
\end{document}